\documentclass{amsart}
\usepackage{amsmath, amssymb, verbatim, amscd, amsthm}
\usepackage[driverfallback=dvipdfm]{hyperref}
\usepackage{color,graphicx,shortvrb}
\usepackage{enumerate}
\usepackage{ulem}

\newtheorem{thm}{Theorem}
\newtheorem{lem}{Lemma}[section] 
\newtheorem{prop}[lem]{Proposition}

\theoremstyle{definition} 
\newtheorem{defn}{Definition} 
\newtheorem{ex}{Example} 
\newtheorem*{exs}{Example} 
\newtheorem{rem}{Remark}

\numberwithin{equation}{section}

 \newcommand{\D}{\displaystyle} 
 \newcommand{\C}{\mathbb C} 
 \newcommand{\N}{\mathbb N} 
 \newcommand{\R}{\mathbb R} 
 \newcommand{\T}{\mathbb T}
 \newcommand{\Di}{\mathbb D}
 \newcommand{\Db}{\overline{\Di}}
 \newcommand{\fr}{\frac} 
 \newcommand{\V}[1]{\left\Vert #1 \right\Vert} 
 \newcommand{\q}{\quad} 
 \newcommand{\qq}{\qquad}

 \newcommand{\al}{\alpha} 
 \newcommand{\ov}{\overline} 
 \newcommand{\la}{\lambda} 
 \renewcommand{\d}{\delta}
 \newcommand{\set}[1]{\{ #1 \}} 
 \renewcommand{\Re}{\operatorname{Re}}
 
 \newcommand{\Ao}{B}
 
 \newcommand{\M}{\cha}
 \newcommand{\Vinf}[1]{\Vert #1 \Vert_\infty}
 \newcommand{\Vs}[1]{\Vert #1 \Vert_B} 
 \newcommand{\VB}[1]{\Vert #1 \Vert_B}
 \newcommand{\tf}{\t{a}}
 \newcommand{\tg}{\t{b}}
 \newcommand{\e}{\varepsilon_1} 
 \newcommand{\et}{\varepsilon_2} 
 
 \newcommand{\phio}{\phi_1}
 
 \newcommand{\unit}{\mathbf{1}}
 \newcommand{\ext}[1]{\operatorname{ext}(#1)}
 \newcommand{\chb}{\operatorname{Ch}(\TB)}
 
 \renewcommand{\t}[1]{\widetilde{#1}}

 \newcommand{\del}{\Delta}
  \newcommand{\deli}{\Delta^{-1}}
 \newcommand{\be}{\beta}
 
 \newcommand{\phit}{\phi_2}
 \newcommand{\SB}{S_{\TB}}
 \newcommand{\SAo}{S_{\Ao}}

 \renewcommand{\dh}{d_H}

 \newcommand{\unitI}{e}
 
 \newcommand{\MT}{\M\times\T}
 \newcommand{\sm}{\setminus}
 \newcommand{\delo}{\Delta_1}
 \newcommand{\INT}{\mathcal{I}}
 \newcommand{\der}{\mathcal{D}}
 \newcommand{\con}{\mathcal{C}}
 \newcommand{\ta}{\t{a}}
 \newcommand{\tb}{\t{b}}
 \newcommand{\te}{\t{e}}
 \newcommand{\TB}{\t{B}}
 \newcommand{\cha}{\operatorname{Ch}(A)}

 \DeclareMathOperator{\Lip}{Lip}
 \newcommand{\Lipo}{\Lip([0,1])}
 \newcommand{\MIS}{\mathcal{M}}

 \newcommand{\SA}{S_A}

\usepackage{ulem}
\usepackage{marginnote}
\usepackage{geometry}

\begin{document}


\title[Tingley's problem for Banach spaces]
{Tingley's problem for complex Banach spaces
which do not satisfy the Hausdorff distance condition}


\author[{\small D. Cabezas}]{David Cabezas}
\address[D. Cabezas]{
Departamento de An{\'a}lisis Matem{\'a}tico, Facultad de
Ciencias, Universidad de Granada, 18071 Granada, Spain.}
\email{dcabezas@ugr.es}

\author[M. Cueto-Avellaneda]{Mar{\'i}a Cueto-Avellaneda}
\address[M. Cueto-Avellaneda]{
School of Mathematics, Statistics and Actuarial Science,
University of Kent, Canterbury, Kent CT2 7NX, UK}
\email{emecueto@gmail.com}

\author[Y. Enami]{Yuta Enami}
\address[Y. Enami]{Graduate School of Science and Technology,
Niigata University, Niigata 950-2181, Japan}
\email{enami@m.sc.niigata-u.ac.jp}

\author[T. Miura]{Takeshi Miura}
\address[T. Miura]{Department of Mathematics,
Faculty of Science, Niigata University, Niigata 950-2181, Japan}
\email{miura@math.sc.niigata-u.ac.jp}

\author[A.M. Peralta]{Antonio M. Peralta}
\address[A.M. Peralta]
{Instituto de Matem{\'a}ticas de la Universidad de Granada (IMAG),
Departamento de An{\'a}lisis Matem{\'a}tico, Facultad de
Ciencias, Universidad de Granada, 18071 Granada, Spain.}
\email{aperalta@ugr.es}

\thanks{
First author supported by grant FPU21/00617 at University of Granada
founded by Ministerio de Universidades (Spain).
First, second, and fifth authors partially supported by grant
PID2021-122126NB-C31 funded by MCIN/AEI/10.13039/501100011033. Third author partially supported by JSPS KAKENHI
Grant Number JP 21J21512. The fourth author was partially supported
by JSPS KAKENHI Grant Number JP 20K03650.
The fifth author also supported
by IMAG--Mar{\'i}a de Maeztu grant
CEX2020-001105-M/AEI/10.13039/ 501100011033.}


\begin{abstract} 
In 2022, Hatori gave a sufficient condition for complex Banach spaces to have the complex Mazur--Ulam property. In this paper, we introduce a class of complex Banach spaces $B$ that do not
satisfy the condition but enjoy the property that every surjective isometry on the
unit sphere of
such $B$ admits an extension to a surjective real linear isometry on the whole space $B$. 
Typical examples of Banach spaces studied in this note are the spaces $\Lipo$ of all
Lipschitz complex-valued functions on $[0,1]$  with norm $\VB{a}=|a(c)|+L(a)$, where $c$ is a fixed element in $[0,1]$, and the space $C^1([0,1])$
of all continuously differentiable complex-valued functions on $[0,1]$ equipped with the 
norm $|f(0)|+\Vinf{f'}$.
\end{abstract}

\maketitle


\section{Introduction and the main results}

Let $S_E, S_F$ denote the unit spheres of two real or complex Banach spaces $E$ and $F$.
In 1987, Tingley \cite{tin} asked whether every surjective isometry
$\del\colon S_E\to S_F$ admits an extension to a surjective real linear isometry from $E$ onto $F$.
This sharp question, nowadays known as {\it Tingley's problem}, has been deeply investigated
for several concrete classes of Banach spaces since then. The problem has become an attractive and active area of research in functional analysis. There is a huge list of affirmative answers to Tingley's problem, see for example \cite{cue,fer,fer2,fer3,fer4,fer5,fer6,hat2,hir,leu2,mor,
per2,per4,per5,tana2,tana3,tana4,wan1,wan2}.
However the problem is still open even for finite dimensional Banach spaces. In fact,
only in 2022 Banakh \cite{ban2} obtained an affirmative answer to Tingley's problem for 2-dimensional Banach spaces, the question remains as an open challenge for Banach spaces with dimension greater than or equal to $3$.

Tingley's problem motivated the study of  the so called Mazur--Ulam property. A real or complex Banach space $E$ has the {\it Mazur--Ulam property}
if for each Banach space $F$, every surjective isometry
$\del\colon S_E\to S_F$ can be extended to a surjective real linear isometry between the whole spaces $E$ and $F$.
Many mathematicians have been involved in the study of the Mazur--Ulam property on several classes of Banach spaces (cf. \cite{cab2,che,cue2,hat,hat3,tan,tan2,WangNiu2022}). A refinement of the Mazur--Ulam property was introduced by Hatori in \cite{hat}. A complex Banach space $E$ is said to have the {\it complex Mazur--Ulam
property}, if for any complex Banach space $F$, every surjective isometry $\del\colon S_E\to S_F$ admits an extension to a surjective real linear isometry between
the whole spaces $E$ and $F$. 
In the just quoted reference, Hatori establishes a sufficient condition to guarantee that a complex Banach space $E$ has the complex Mazur--Ulam property.
The condition is related to the Hausdorff distance between any two
maximal convex subsets of $S_E$.

The purpose of this paper is to solve Tingley's problem for a class of complex Banach spaces which do not satisfy the condition on the Hausdorff distance required in Hatori's theorem \cite{hat}.
The main result of this paper is applied to several  classic complex Banach spaces,
including the Banach space $\Lipo$ of all Lipschitz continuous complex-valued functions on $[0,1]$ with norm $\VB{a}=|a(c)|+L(a)$, where $c$ is a fixed element in $[0,1]$, and the space $C^1([0,1])$ of all continuously differentiable complex-valued functions on $[0,1],$ with respect to the norm defined by  $\|f\| =|f(0)|+\sup_{t\in[0,1]}|f'(t)|$.

Let $X$ be a locally compact Hausdorff space and $C_0(X)$ the Banach space of all complex-valued continuous functions on $X$ which vanish at infinity,
equipped with the supremum norm $\Vinf{f}=\sup_{x\in X}|f(x)|$ for $f\in C_0(X)$. We write $C(X),$ instead of $C_0(X),$ in case that $X$ is compact. 
Let $A$ be a complex normed linear subspace of $C_0(X)$. For each $x \in X$, we define $\d_{x} : A \to \C$ by $\d_{x}(f) = f(x)$ for every $f \in A$.
The Choquet boundary of $A$, $\cha,$ is the set of all points $x \in X$ for which $\d_x$ becomes
an extreme point of the unit ball, $A_1^*,$ of
the dual space, $A^*,$ of $A$. 
The subspace $A$ is said to be {\it extremely C-regular}
if for each $\varepsilon>0$, $x\in\cha$, and each open neighborhood $O$ of $x$ in $X$ there exists $f\in\SA$ such that $f(x)=1$ and $|f|<\varepsilon$ on $X\sm O$.

Briefly speaking, the fundamental theorem of calculus states that the equality
$\fr{d}{dt}\int_0^tf(s)\,ds=f(t)$ holds for all $t\in[0,1]$ and $f\in C([0,1])$.
We may regard the integral, $\int_0^t,$ and derivative, $\fr{d}{dt},$ as complex linear operators from $C([0,1])$ to $C^1([0,1])$ and from $C^1([0,1])$ to $C([0,1])$, respectively. Then $C^1([0,1])$ is a Banach space with respect to the norm $\|f\| =|f(0)|+\Vinf{f'}$ for $f\in C^1([0,1])$. Motivated by the space $C^1([0,1])$ with the mentioned operators,
we introduce a class of complex normed spaces, including $C^1([0,1])$.

\begin{thm}\label{thm1}
Let $B$ be a complex normed space with norm
$\VB{\cdot}$, for which there exist 
an extremely C-regular subspace $A$ of $C_0(X)$,
bounded complex linear operators
$\INT\colon A\to B$, $\der\colon B\to A,$
and a bounded linear functional
$\con\colon B\to\C$ with the following properties:
\begin{enumerate}[{(i)}]
\item\label{(1)}
$(\der\circ\INT)(f)=f,$ for every $f\in A$,
\item\label{(4)}
$(\con\circ\INT)(f)=0,$ for every $f\in A$,
\item\label{(5)}
$\VB{a}=|\con(a)|+\Vinf{\der(a)},$ for every $a\in B$,
\item\label{(6)}
there exists $e\in\ker\der$, the kernel of $\der$, with $\con(e)=\VB{e}=1$.
\end{enumerate}
If $\del\colon\SAo\to\SAo$ is a surjective isometry with respect to the norm $\Vs{\cdot}$,
then $\Delta$ extends to a surjective real linear isometry on $\Ao$.
\end{thm}

In order to understand the significance of Theorem~\ref{thm1}, we list some examples of Banach spaces $A$ and $B$ satisfying the hypotheses of the theorem--in section~\ref{sect6}, we shall discuss them in detail.

\begin{exs} As we shall see in section \ref{sect6}, the normed spaces $(B,\VB{\cdot})$ in the following list satisfy all the hypotheses $(\ref{(1)})$ through $(\ref{(6)})$ in Theorem~\ref{thm1}, although they are not, in general, uniform algebras (recall that unital uniform algebras can be characterized as those unital commutative Banach algebras satisfying $\|a^{2}\|=\|a\|^{2}$),  and hence they are out from the scope of the results by Hatori, Oi and Shindo Togashi.

\begin{enumerate}[$(1)$]
\item
Let $B$ be the complex linear space $\Lipo$ of all Lipschitz (continuous) complex-valued functions with the norm $\VB{a}=|a(c)|+L(a)$ for $a\in B$, where $c\in[0,1]$ and
$L(a)$ is the Lipschitz constant for $a$.

\item
Let $B$ be the Banach space $C^1([0,1])$ of all continuously differentiable complex-valued functions on $[0,1]$
with the norm $\VB{a}=|a(c)|+\Vinf{a'}$ for $a\in B$, where $c\in[0,1]$.

\item
Let $H(\Di)$ be the algebra of all analytic functions on the open unit disk $\Di$
and let $H^\infty(\Di)$ be the commutative Banach algebra
of all bounded analytic functions on $\Di$.
We define
$B=\set{a\in H(\Di):a'\in H^\infty(\Di)}$
with the norm $\VB{a}=|a(c)|+\Vinf{a'}$ for $a\in B$,
where $c\in\Di$.

\item
Let $A$ be the commutative Banach algebra $A(\Db)$ of all analytic functions
on $\Di$
which can be extended to
continuous functions on
the closed unit disk $\Db$.
Set $B=\set{a\in H(\Di):a'\in A\ \mbox{and}\ a'(0)=0}$
with $\VB{a}=|a(c)|+\Vinf{a'}$ for $a\in B$, where $c\in\Di$.

\end{enumerate}
\end{exs}

It is also worth to note that a recent result by Hatori provides sufficient conditions on a complex Banach space $E$ (one of them is the Hausdorff distance condition) to conclude that it satisfies  the complex Mazur--Ulam property (see \cite[Proposition~4.4.]{hat}).  We shall see in Remark~\ref{rem1} that even under the stronger hypothesis that $X$ is a compact Hausdorff space,  each complex Banach space $B$ satisfying the hypotheses in Theorem~\ref{thm1} fails the Hausdorff distance condition.

We also observe that any normed space $B$ under the hypotheses in Theorem~\ref{thm1} is
isometrically isomorphic to the $\ell_1$-direct sum, $A\oplus^{\ell_1} \mathbb{C},$ of an extremely C-regular subspace $A \subseteq C_0(X)$ and $\mathbb{C}$ (the mapping $B \ni a \mapsto (\mathcal{D}(a), \mathcal{C}(a)) \in  A\oplus^{\ell_1} \mathbb{C}$ defines a surjective linear isometry). Tingley's problem has been treated in the literature in the particular case of normed spaces which can be expressed as $\ell_1$-sums. For example, Wang and Orihara found a positive solution to Tingley's problem in the case of a surjective isometry between the unit spheres of two normed spaces which coincide with the $\ell_1$-sum of two families of strictly convex normed spaces (see \cite{WanOrihara2002}). The same authors also solved Tingley's problem in the case in which $E$ and $F$ are $\ell_1$-sums of spaces of the form $C_0(L,G)$, where $L$ is a locally compact Hausdorff space and $G$ is a strictly convex real normed space, and the number of involved summands counting both decompositions is $\geq 3$ (cf. \cite[Theorem 3.5]{wan2}). Moreover, let $\{G_i\}_{i\in \Gamma}$ be a family of strictly convex normed spaces and let $\Delta$ be an isometry from the unit sphere of $G=\bigoplus_{i\in \Gamma}^{\ell_1} G_i$ into the unit sphere of another normed space $E$. Li found in \cite[Theorem 2.7]{Li2016} sufficient conditions to assure that $\Delta$ extends to an isometry from $G$ to $E$. In our result, $B\cong  A\oplus^{\ell_1} \mathbb{C}$, where $A$ being an extremely C-regular subspace of some $C_0(X)$ is a too weak hypothesis to apply any of the previous results by Wang, Orihara and Li.  The techniques here follow rather new ideas.

\section{Preliminaries and auxiliary lemmata}

In Theorem~\ref{thm2}, we first consider a special (we assume completeness
of the norm and compactness of the topological space $X$),
but essential, case of Theorem~\ref{thm1}. A proof of the general statement of Theorem~\ref{thm1} will be given in section~\ref{sect5}.

\begin{thm}\label{thm2}
Let $B$ be a complex Banach space with norm $\V{\cdot}_B$,
for which there exist an extremely C-regular closed subspace $A$
of $C(X)$ for some compact Hausdorff space $X$, bounded linear operators $\INT\colon A\to B$,
$\der\colon B\to A$ and $\con\colon B\to\C$ satisfying the hypotheses $(\ref{(1)})$
to $(\ref{(6)})$ in Theorem~\ref{thm1}. Let $\del\colon S_B\to S_B$ be a surjective isometry with respect to the norm
$\V{\cdot}_B$. Then $\Delta$ extends to a surjective real linear isometry on $B$.
\end{thm}

In Remark~\ref{rem1}, just before Lemma~\ref{lem3.6},
we show that no complex Banach space $B$ enjoying the hypotheses in Theorem~\ref{thm2} satisfies the Hausdorff distance condition introduced by Hatori in \cite{hat}.

The main idea behind the proof of Theorem~\ref{thm2} consists in embedding the complex Banach space $B$ into $C(Y)$ for some compact Hausdorff space $Y$. Let $\TB$ be the isometric image of $B$ in $C(Y)$. If $\del\colon\SAo\to\SAo$ is a surjective isometry, then $\del$ induces a surjective isometry $\t{\del}\colon\SB\to\SB$.
We can characterize the set of all maximal convex subsets of $\SB$.
Combining the ideas in \cite{cue,hat2}, we obtain enough information about the isometry $\t{\del}$ on $\SB$.
Since $\t{\del}$ is induced by $\del$ on $\SAo$, we can determine $\del$ from the properties of $\t{\del}$.

In the rest of this section, we assume that $A$ and $B$ satisfy the hypotheses in Theorem~\ref{thm2} for the bounded linear operators $\INT$, $\der$ and $\con$.
We denote by $\T$ the unit circle in the field of complex numbers  $\C$.
For each $a\in B$ and $(x,z)\in X\times\T$, we define $\ta(x,z)$ by
\begin{equation}\label{tilde}
\ta(x,z)=\con(a)+\der(a)(x)z.
\end{equation}
Since $\der(a)$ belongs to $A$, we see that
the function $\ta$ is continuous on $X\times\T$ with
respect to the product topology.
We set 
\[
\TB=\set{\ta\in C(X\times\T):a\in B}.
\]
Then $\TB$ is a closed linear subspace of $C(X\times\T)$ with the supremum norm $\Vinf{\cdot}$ on $X\times\T$ (we note that $\TB$ is closed because the norm
$\Vs{\cdot}$ is complete).  We define a mapping $U \colon (\Ao, \Vs{\cdot}) \to (\TB,\Vinf{\cdot})$ given by
\[
U(a)=\ta
\qq(a\in\Ao).
\]
We claim that $U$ is a surjective complex linear map from $\Ao$ onto $\TB$. Namely, by \eqref{tilde}, $\Vinf{U(a)}=\Vs{a}$ holds for all $a\in\Ao$:
In fact, for each $a\in\Ao$, there exist $z_0,z_1\in\T$ and $x_0\in X$
such that $\con(a)=|\con(a)|z_0$ and $\der(a)(x_0)=\Vinf{\der(a)}z_1$.
Then
\begin{align*}
|U(a)(x_0,z_0\ov{z_1})|
&=
|\con(a)+\der(a)(x_0)z_0\ov{z_1}|
=|(|\con(a)|+\Vinf{\der(a)})z_0|\\
&=
|\con(a)|+\Vinf{\der(a)}=\VB{a}.
\end{align*}
We thus obtain $\VB{a}\leq\Vinf{U(a)}$.
For each $(x,z)\in X\times\T$, we have
\begin{align*}
|U(a)(x,z)|
&=
|\con(a)+\der(a)(x)z|\leq|\con(a)|+|\der(a)(x)|\\
&\leq
|\con(a)|+\Vinf{\der(a)}=\VB{a},
\end{align*}
which yields $\Vinf{U(a)}\leq\VB{a}$.
Consequently,
\begin{equation}\label{norm}
\Vinf{\ta}=\Vinf{U(a)}=\VB{a}\qq(a\in\Ao).
\end{equation}
Therefore, the map $U$ is a surjective complex linear isometry
from $(\Ao,\VB{\cdot})$ onto $(\TB,\Vinf{\cdot})$.
In particular, $U(\SAo)=\SB$.

Here, we note that $\INT(f)\in\SAo$ for $f\in\SA$. In fact,
\[
\VB{\INT(f)}=|\con(\INT(f))|+\Vinf{\der(\INT(f))}
=\Vinf{f}=1
\]
by hypotheses $(\ref{(1)})$ and $(\ref{(4)})$.
This implies $\INT(\SA)\subset\SAo$. 

The purpose of this paper is to prove that every surjective isometry
$\del$ on $(\SAo,\Vs{\cdot})$ can be extended to a surjective real linear isometry on $(\Ao,\Vs{\cdot})$.
The complex linear isometry $U$, defined as above,
embeds $B$ into $C(X\times\T)$ with $U(\SAo)=\SB$.
We see that the mapping $U\Delta U^{-1}$
is a well defined, surjective isometry from $(\SB, \Vinf{\cdot})$
onto itself.
We set $T=U\del U^{-1}$ and investigate its properties.
\[
\begin{CD} 
\SAo@>{\Delta}>>\SAo\\ 
@V{U}VV
@VV{U}V\\ 
\SB@>>{T}>\SB
\end{CD} 
\]
By definition $TU=U\Delta$ holds, which is written as
\begin{equation}\label{T}
T(\ta) = \widetilde{\Delta(a)} \qq (a \in \SAo).
\end{equation}
Now we recall that Choquet boundary for $A$, denoted by $\cha$,
is the set of all points $x\in X$ so that the point evaluation $\d_x$,
defined by $\d_x(f)=f(x)$ for $f\in A$, is an extreme point
of the closed unit ball, $A_1^*,$ of the dual space of $A$. For each $\la\in\T$ and $\eta\in\cha\times\T$, we define the subset $\la V_\eta$ of $\SB$ as
\begin{equation}\label{V}
\la V_\eta=\set{\tf\in\SB:\tf(\eta)=\la}.
\end{equation}
The properties $\con(e)=1$ and $\der(e)=0$
of $e\in\Ao$ combined with the definition in \eqref{tilde} assure that
\[
\te\equiv 1
\q\mbox{on}\q
X\times\T.
\]
The next lemma plays an important role when we prove
the uniqueness of maximal convex subsets.

\begin{lem}\label{lem2.1} Let $(\la_1,\eta_1),(\la_2,\eta_2)\in\T\times(\M\times\T)$.
If $\la_1V_{\eta_1}\subset\la_2V_{\eta_2}$, then
$(\la_1,\eta_1)=(\la_2,\eta_2)$.
\end{lem}

\begin{proof}
Let $e\in\Ao$ be the element given by the hypothesis $(\ref{(6)})$
 in Theorem~\ref{thm2}. 
In this case we have $\te\in V_{\eta_1}$, since $\te\equiv 1$ on $X\times\T$, and hence $\te(\eta_1)=1=\te(\eta_2)$.
By the assumption, $\la_1\te\in\la_1V_{\eta_1}\subset\la_2V_{\eta_2}$,
which implies $\la_1\te(\eta_2)=\la_2$.
These equalities show that $\la_1=\la_1\te(\eta_2)=\la_2$, and thus
$\la_1=\la_2$.

Since $\eta_j\in\MT$, there exist $x_j\in\cha$ and $z_j\in\T$
such that $\eta_j=(x_j,z_j)$ for $j=1,2$.
We shall next prove that $x_1=x_2$. Otherwise, since $A$ is extremely C-regular, we can find $f\in\SA$ so that
$f(x_1)=1$ and $|f(x_2)|<1$. Setting $a=\mathcal{I}(\la_1\ov{z_1}f)\in B$, we obtain $\con(a)=0$ and $\der(a)=\la_1\ov{z_1}f$ by hypotheses $(\ref{(1)})$ and $(\ref{(4)})$, and thus $a=\mathcal{I}(\la_1\ov{z_1}f)\in\SAo$.
Applying \eqref{tilde}, we obtain $\tf(x,z)=\der(a)(x)z=\la_1\ov{z_1}f(x)z$
for all $(x,z)\in X\times\T$.
By the choice of $f\in\SA$, we have $\tf(x_1,z_1)=\la_1$ and $|\tf(x_2,z_2)|<1$.
This contradicts the assumption that $\la_1V_{\eta_1}\subset\la_2V_{\eta_2}$,
since $\tf\in\la_1V_{\eta_1}\setminus\la_2V_{\eta_2}$ with $\eta_j=(x_j,z_j)$.
We have proved that $x_1=x_2$.

We shall finally show that $z_1=z_2$. By applying that $A$ is extremely C-regular, we can find $f_1\in S_{A}$ satisfying $f_1 (x_1) =1$ (and an additional property not required here). We set $B\ni b=\,\mathcal{I}(\la_1 \ov{z_1}\,f_1)$. According to hypotheses $(\ref{(1)})$ and $(\ref{(4)})$, we get $\con(b)=0$ and $\der(b)=\la_1\ov{z_1}\,f_1$. Then $b$ satisfies the following properties: $\tg\in\SB$ ($b\in S_B$) and $\tg(x_1,z_1)=\la_1$, that is, $\tg\in\la_1V_{\eta_1}$.
	We deduce from the hypothesis of the lemma that $\tg\in\la_2V_{\eta_2}$, which implies
	that $\la_2=\tg(x_2,z_2)=\la_1\ov{z_1}z_2$ by the choice of $b$ and the fact that $x_1=x_2$.
	Having in mind that $\la_1=\la_2$, we conclude that $z_1=z_2$. Consequently, $\eta_1=(x_1,z_1)=(x_2,z_2)=\eta_2$,
	and hence $(\la_1,\eta_1)=(\la_2,\eta_2)$.
\end{proof}

In order to determine the maximal
convex subsets of $S_{\TB}$, we need to characterize the extreme points of the closed unit ball, $B^*_1,$ of the dual space, $B^*,$ of $B$. For each $\eta\in X\times\T$, let $\d_\eta\colon\TB\to\C$ be the point evaluation
functional on $C(X\times \mathbb{T})$, defined by $\d_\eta(f)=f(\eta)$ for every $f\in C(X\times \mathbb{T})$. Since the Hahn--Banach extension of a functional on a closed subspace is not, in general, unique, there are, in principle, many Hahn--Banach extensions of each functional in $\tilde{B}^*$. For each $\phi$ in the closed unit ball of $\tilde{B}^*,$ by the Hahn--Banach theorem and the classical Riesz representation theorem, there exists a (non-necessarily unique) Borel regular measure $\sigma$ on $X\times \mathbb{T}$ such that $\phi (h)=\int_{X\times\T}h\,d\sigma$ for all $h\in\TB$ and the total variation, $\V{\sigma}$, of $\sigma$ satisfies $\V{\sigma}=\|\phi\|$. Let us consider an extreme point of the closed unit ball of $C(X\times \mathbb{T})^*$ of the form $\delta_{\eta}$ with $\eta\in X\times \mathbb{T},$ the restricted functional $\delta_{\eta}|_{\tilde{B}}$ and the set of all Borel regular measures $\sigma$ representing each one of the different Hahn-Banach extensions of $\delta_{\eta}|_{\tilde{B}}$. Each one of these Borel regular measures will be called a {\it representing measure} for $\d_\eta$.

\begin{lem}\label{lem2.2}
Let $\sigma$ be a representing measure for
$\d_{\eta_0}|_{\tilde{B}}$ with $\eta_0=(x_0,z_0)$ in $\M\times\T$.
Then $\sigma$ is the Dirac measure concentrated at $\eta_0$. In other words the Hahn--Banach extension of $\d_{\eta_0}|_{\tilde{B}}$ to $C(X\times \mathbb{T})^*$ is unique and coincides with $\d_{\eta_0}$.
\end{lem}

\begin{proof}
Fix a positive $\varepsilon$ and take any open neighbourhood $O$ of $x_0$ in $X$.
By the extreme C-regularity of $A$ there exists $f\in\SA$
satisfying $f(x_0)=1$ and $|f|<\varepsilon$ on $X\setminus O$.
Here, we recall that $\te=1$ on $X\times\T$.
We deduce from $\V{\d_{\eta_0}|_{\tilde{B}}}=1=\d_{\eta_0}(\te)$ that
$\sigma$ is a probability measure on $X\times\T$ (see \cite[p.81]{bro}).
By \eqref{tilde} and the hypotheses $(\ref{(1)})$ and $(\ref{(4)})$, we see that the element $a=\mathcal{I}(f)\in\SAo$ satisfies $\tf(x,z)=f(x)z$
for all $(x,z)\in X\times\T$.
In particular, we have $\tf(\eta_0)=z_0$ by the choice of $f\in A$.
The property $|f|<\varepsilon$ on $X\sm O$ ensures that
$|\tf|<\varepsilon$ on $(X\sm O)\times\T$.
It follows from the above properties of $\tf$ that
\begin{align*}
1
&=
|z_0|=
|\d_{\eta_0}(\tf)|
=\left|\int_{X\times\T}\tf\,d\sigma\right|
\leq
\left|\int_{O\times\T}\tf\,d\sigma\right|
+\left|\int_{(X\sm O)\times\T}\tf\,d\sigma\right|\\
&\leq
\int_{O\times\T}|\tf|\,d\sigma
+\int_{(X\sm O)\times\T}|\tf|\,d\sigma
\leq
\Vinf{\tf}\sigma(O\times\T)+\varepsilon\sigma((X\sm O)\times\T)\\
&\leq
\sigma(O\times\T)+\varepsilon\sigma(X\times\T)
=\sigma(O\times\T)+\varepsilon\V{\sigma}=\sigma(O\times\T)+\varepsilon.
\end{align*}
We obtain $1\leq\sigma(O\times\T)\leq\sigma(X\times\T)=\V{\sigma}=1$,
since $\varepsilon>0$ is arbitrary.
Therefore, $\sigma(O\times\T)=1$ for all open sets $O$ in $X$
with $x_0\in O$, and consequently, we observe that
$\sigma(\set{x_0}\times\T)=1$ by the regularity of the measure $\sigma$.
We derive from $\sigma(\set{x_0}\times\T)=1=\sigma(X\times\T)$ that
\[
z_0=
\d_{\eta_0}(\tf)=\int_{\set{x_0}\times\T}\tf\,d\sigma
=\int_{\set{x_0}\times\T}f(x)z\,d\sigma
=\int_{\set{x_0}\times\T}z\,d\sigma,
\]
which yields $\int_{\set{x_0}\times\T}(z_0-z)\,d\sigma=0$.
Taking $Z=\set{x_0}\times(\T\sm\set{z_0})$, we get
$\int_Z(1-\ov{z_0}z)\,d\sigma
=-\ov{z_0}\int_Z(z-z_0)\,d\sigma=0$.
Since $\sigma$ is a positive measure on $X\times\T$,
we have $\int_Z\Re(1-\ov{z_0}z)\,d\sigma=0$.
Having in mind that $\Re(1-\ov{z_0}z)>0$ on $Z$, we conclude $\sigma(Z)=0$.
We derive from $\sigma(\set{x_0}\times\T)=1$ and $\sigma(Z)=0$ that
$\sigma(\set{x_0}\times\set{z_0})=1$, which shows that
$\sigma$ is the Dirac measure concentrated at $\eta_0=(x_0,z_0)$.
\end{proof}

We denote by $\ext{\TB_1^*}$ the set of all extreme points of the closed unit ball of $\TB_1^*.$
Let $\chb$ be Choquet boundary for $\TB$.
The Arens-Kelley theorem (cf. \cite[Corollary~2.3.6]{fle1}) states that
$\ext{\TB_1^*}=\set{\la\d_\eta\in\TB_1^*:\la\in\T,\,\eta\in\chb}$. We shall next characterize $\chb$ to determine $\ext{\TB_1^*}$.

\begin{lem}\label{lem2.3}
The Choquet boundary, $\chb,$ of $\TB$ in $X\times \mathbb{T}$ is of the form $\chb=\MT$.
\end{lem}

\begin{proof}
We first prove $\MT\subset\chb$. It is part of the folklore in functional analysis that if $\varphi \in \ext{E_1^*}$ for some Banach space $E$ and $F$ is a closed subspace of $E$ such that $\varphi|_{F}$ admits a unique Hahn-Banach extension to $E$, the functional $\varphi|_{F}$ is an extreme point of $F_1^*$.  Lemma \ref{lem2.2} then assures that $\MT\subset\chb$.

To see the converse implication, we observe that the Banach space $\tilde{B}$ is isometrically isomorphic to the $\ell_1$-direct sum, $A\oplus^{\ell_1} \mathbb{C}$ via the surjective linear isometry given by $\tilde{B} \ni \tilde{a} \mapsto (\mathcal{D}(a), \mathcal{C}(a)) \in  A\oplus^{\ell_1} \mathbb{C}$, and therefore $\tilde{B}^*$ is isometrically isomorphic to $A^*\oplus^{\ell_\infty} \mathbb{C}$ via the transposed surjective isometry. It is well known that $$\ext{\left(A^*\oplus^{\ell_\infty} \mathbb{C}\right)_1} = \left\{ (\phi,z) : \phi\in \ext{A_1^*}, \, z\in \mathbb{T} \right\}.$$ Fix an arbitrary $\eta_1 = (x_1,z_1)\in\chb\subseteq X\times \mathbb{T}$, and the corresponding  $\d_{\eta_1}|_{\tilde{B}}\in\ext{\TB_1^*}$. The previous identifications result, via the Arens--Kelley theorem, in the form $\d_{\eta_1}|_{\tilde{B}} (\tilde{a}) = \alpha \delta_x (\der(a)) + \beta\,\con(a),$ ($\forall \tilde{a}\in\tilde{B}$) for some $\alpha,\beta \in\mathbb{T}$ and $x\in \cha$. Having in mind that $\tilde{e}$ is the constant function on $X\times \mathbb{T}$ with $\der(e) =0$, we get $1 =\d_{\eta_1}|_{\tilde{B}} (\tilde{e}) = \beta$. Finally, for $\widetilde{\INT(f)} \in\tilde{B}$ with an arbitrary $f\in A$, we obtain $$z_1 f(x_1) = \widetilde{\INT(f)} (\eta_1) =\d_{\eta_1}|_{\tilde{B}} (\widetilde{\INT(f)}) =  \alpha f(x). $$ The arbitrariness of $ f\in A,$ the hypothesis affirming that $A$ is extremely C-regular, and the fact that $x\in \cha$ imply $x_1 =x\in \cha$ and $z_1 = \alpha$.
\end{proof}

A subtle combination of Eidelheit's or Hahn-Banach separation theorem, due to Tanaka \cite[Lemma 3.3]{tana2} and Hatori, Oi and Shindo Togashi \cite[Lemma~3.1]{hat2}, proves that each maximal convex subset $F$ of the unit sphere of a Banach space $E$ writes in the form 
$F=\xi^{-1}(1)\cap S_{E}$, where $\xi$ is an extreme point of the closed unit ball of $E^*$ and  $\xi^{-1}(1)=\set{a\in S_{E} :\xi(a)=1}$.
We can prove the following result by a quite similar argument as in \cite[Proof of Lemma~3.2]{hat2}. 
The proof is included here for completeness.

\begin{prop}\label{prop2.4}
Let $F$ be a subset of $\SB$.
Then $F$ is a maximal convex subset of $\SB$
if and only if there exist $\la\in\T$
and $\eta\in\M\times\T$ such that $F=\la V_\eta$.
\end{prop}

\begin{proof}
Suppose that $F$ is a maximal convex subset of $\SB$.
As we noted above, there exists 
$\xi\in\ext{\TB_1^*}$ such that $F=\xi^{-1}(1)\cap\SB$ (cf. \cite[Lemma~3.1]{hat2}).
Since, by Lemma~\ref{lem2.3}, $\chb=\MT$ we can write $$\ext{\TB_1^*}=\left\{\la\d_\eta\in\TB_1^*:\la\in\T,\,\eta\in\MT\right\}.$$ This implies that $\xi=\la\d_\eta$ for some $\la\in\T$ and $\eta\in\MT$.
Then, by \eqref{V}, the maximal convex set $F$ is of the form
\[
F=(\la\d_\eta)^{-1}(1)\cap\SB=\set{\tf\in\SB:\la\tf(\eta)=1}=\ov{\la}V_\eta.
\]

Conversely, suppose that there exist $\la_1\in\T$ and $\eta_1\in\MT$ such that $F=\la_1V_{\eta_1}$.
We shall see that $F$ is a convex subset of $\SB$.
Zorn's lemma assures that there exists a maximal convex subset $K$ of $\SB$
such that $F\subset K$.
Then, by what we proved in the first part, $K=\la_2V_{\eta_2}$ for some $\la_2\in\T$ and $\eta_2\in\MT$.
This implies $\la_1V_{\eta_1}=F\subset K=\la_2V_{\eta_2}$, and thus
$(\la_1,\eta_1)=(\la_2,\eta_2)$ by Lemma~\ref{lem2.1}.
We derive $F=K$, and therefore, $F$ is a maximal convex subset of $\SB$.
\end{proof}

\section{Construction of maps $\al$ and $\phi$}

It is known that every surjective isometry between the unit spheres of two Banach spaces
preserves maximal convex subsets of the corresponding spheres (cf. \cite[Lemma~5.1]{che} and 
\cite[Lemma~3.5]{tana}).
Let $E_1$ and $E_2$ be Banach spaces and let $W\colon S_{E_1}\to S_{E_2}$ be
a surjective isometry between the unit spheres $S_{E_1}$ and $S_{E_2}$
of $E_1$ and $E_2$, respectively.
Mori \cite[Proposition~2.3]{mor} proved that $W(-F)=-W(F)$ for each maximal convex
set $F$ of $S_{E_1}$. The following lemma is a direct consequence of the just commented results by Cheng and Dong \cite{che},
Tanaka \cite{tana} and Mori \cite{mor}. 
We keep the notation from the previous section, that is, $\del$ is a surjective isometry on $(\SAo,\Vs{\cdot})$ and $T$ is the corresponding surjective isometry induced on $(\SB, \Vinf{\cdot})$.

\begin{lem}\label{lem3.1}
There exist surjective maps $\al\colon\T\times(\MT)\to\T$ and
$\phi\colon\T\times(\MT)\to\MT$ such that
\begin{align}
T(\la V_\eta)
&=
\al(\la,\eta)V_{\phi(\la,\eta)},
\label{lem3.1.1}\\
\al(-\la,\eta)
&=
-\al(\la,\eta)
\label{lem3.1.2}
\end{align}
and $\phi(-\la,\eta)=\phi(\la,\eta)$ for all $(\la,\eta)\in\T\times(\MT)$.
\end{lem}

\begin{proof}
Fix arbitrary $(\la,\eta)\in\T\times(\MT)$.
Proposition~\ref{prop2.4} shows that $\la V_\eta$ is a maximal convex subset of $\SB$.
Since $T$ is a surjective isometry on $\SB$, it preserves maximal convex subsets
of $\SB$ by \cite[Lemma~5.1]{che} and \cite[Lemma~3.5]{tana}.
This implies that $T(\la V_\eta)$ is a maximal convex subset of $\SB$.
Applying Proposition~\ref{prop2.4} again, we can write
$T(\la V_\eta)=\mu_0V_{\zeta_0}$ for some $(\mu_0,\zeta_0)\in\T\times(\MT)$.
We shall prove the uniqueness of such element $(\mu_0,\zeta_0)$.
Assume that $T(\la V_\eta)=\mu_1V_{\zeta_1}$ for another element
$(\mu_1,\zeta_1)\in\T\times(\MT)$.
Then we have $\mu_0V_{\zeta_0}=T(\la V_\eta)=\mu_1V_{\zeta_1}$.
We derive from Lemma~\ref{lem2.1} that $(\mu_0,\zeta_0)=(\mu_1,\zeta_1)$,
which proves the uniqueness of $(\mu_0,\zeta_0)$ with
$T(\la V_\eta)=\mu_0V_{\zeta_0}$.
Since $(\la,\eta)\in\T\times(\MT)$ is arbitrary, the correspondences
$\al(\la,\eta)=\mu_0$ and $\phi(\la,\eta)=\zeta_0$ give two well defined maps
$\al\colon\T\times(\MT)\to\T$ and $\phi\colon\T\times(\MT)\to\M\times\T$
satisfying $T(\la V_\eta)=\al(\la,\eta)V_{\phi(\la,\eta)}$
for all $(\la,\eta)\in\T\times(\MT)$.

Since $\la V_\eta$ is a maximal convex subset of $\SB$, \cite[Proposition~2.3]{mor}
shows that $T(-\la V_\eta)=-T(\la V_\eta)$.
Applying \eqref{lem3.1.1}, we can rewrite the last equality as
$\al(-\la,\eta)V_{\phi(-\la,\eta)}=-\al(\la,\eta)V_{\phi(\la,\eta)}$.
We obtain $\al(-\la,\eta)=-\al(\la,\eta)$ and $\phi(-\la,\eta)=\phi(\la,\eta)$
by Lemma~\ref{lem2.1}.
Having in mind that $T^{-1}$ is a surjective isometry on $\SB$,
we derive the existence of two maps $\be\colon\T\times(\MT)\to\T$ and
$\psi\colon\T\times(\MT)\to\MT$ satisfying
$T^{-1}(\mu V_\zeta)=\be(\mu,\zeta)V_{\psi(\mu,\zeta)}$
for all $(\mu,\zeta)\in\T\times(\MT)$.
By combining the last equality with \eqref{lem3.1.1},
for each $(\mu,\zeta)\in\T\times(\MT)$ we get
\begin{align*}
\mu V_\zeta
&=
T(T^{-1}(\mu V_\zeta))
=T(\be(\mu,\zeta)V_{\psi(\mu,\zeta)})\\
&=
\al(\be(\mu,\zeta),\psi(\mu,\zeta))V_{\phi(\be(\mu,\zeta),\psi(\mu,\zeta))}.
\end{align*}
Another application of Lemma~\ref{lem2.1}  shows that
$\mu=\al(\be(\mu,\zeta),\psi(\mu,\zeta))$
and $\zeta=\phi(\be(\mu,\zeta),\psi(\mu,\zeta))$, which ensures that
both $\al$ and $\phi$ are surjective.
\end{proof}

Let $\phi\colon\T\times(\MT)\to\MT$ be the map defined in Lemma~\ref{lem3.1},
and let $\pi_j$ be the natural projection defined on $\MT$ to the $j$-th coordinate
for $j=1,2$. We define $\phi_j=\pi_j\circ\phi$ for $j=1,2$;
then $\phio\colon\T\times(\MT)\to\M$ and $\phit\colon\T\times(\MT)\to\T$.
We can therefore write
\[
\phi(\la,\eta)=(\phio(\la,\eta),\phit(\la,\eta))
\qq
((\la,\eta)\in\T\times(\MT)).
\]
It follows from Lemma~\ref{lem3.1} that $\phio$ and $\phit$ are
both surjective maps with
\begin{equation}\label{lem2.6.1}
\phi_j(-\la,\eta)=\phi_j(\la,\eta)
\qq((\la,\eta)\in\T\times(\M\times\T),\ j=1,2).
\end{equation}

In the rest of this section, we will analyse the properties of the maps $\al$ and $\phi$
defined on $\T\times(\MT)$; roughly speaking, we shall try to express $\al(\la,x,z)$ and
$\phi_j(\la,x,z)$ as single variable maps with $\la,z\in\T$ and $x\in\cha$.

\begin{lem}\label{lem3.2}
The function $\phio$ satisfies $\phio(\la,\eta)=\phio(1,\eta)$
for all $\la\in\T$ and $\eta\in\MT$.
\end{lem}

\begin{proof}
Choose $\la\in\T$ and $\eta\in\M\times\T$ arbitrarily.
Assume that $\phio(\la,\eta)\neq\phio(1,\eta)$ to lead to a contradiction.
We can therefore find disjoint open sets $O_\la$ and $O_1$ in $X$ so that
$\phio(\nu,\eta)\in O_\nu$ for $\nu=\la,1$.
Fix $\nu\in\set{\la,1}$. Then, by the extreme C-regurality of $A$, there exists $f_\nu\in\SA$ such that
$f_\nu(\phio(\nu,\eta))=1$ and $|f_\nu|<1/4$ on $X\sm O_\nu$.
We denote $\INT(\al(\nu,\eta)\ov{\phit(\nu,\eta)}f_\nu)\in S_{{B}}$ by $a_\nu$.
By \eqref{tilde} combined with hypotheses $(\ref{(1)})$ and $(\ref{(4)})$,
we have $\t{a_\nu}(x,z)=\al(\nu,\eta)\ov{\phit(\nu,\eta)}f_\nu(x)z$
for all $(x,z)\in X\times\T$.
Since $|f_\nu|<1/4$ on $X\sm O_\nu$ with $O_\la\cap O_1=\emptyset$, we get
\[
|\t{a_\la}(x,z)-\t{a_1}(x,z)|\leq|f_\la(x)|+|f_1(x)|<1+\fr{1}{4}=\fr{5}{4}
\]
for all $(x,z)\in X\times\T$, which shows $\Vinf{\t{a_\la}-\t{a_1}}<5/4$.

Now by applying that $\phi(\nu,\eta)=(\phio(\nu,\eta),\phit(\nu,\eta))$,
we see that $\t{a_\nu}$ lies in $\al(\nu,\eta)V_{\phi(\nu,\eta)}=T(\nu V_\eta)$
by \eqref{lem3.1.1}.
Hence $T^{-1}(\t{a_\nu})\in\nu V_\eta$, which implies
$T^{-1}(\t{a_\la})(\eta)=\la$ and $T^{-1}(\t{a_1})(\eta)=1$.
First, let us consider the case in which $\Re\la\leq0$.
Then we obtain
\begin{align*}
\sqrt{2}
&\leq
|\la-1|=|T^{-1}(\t{a_\la})(\eta)-T^{-1}(\t{a_1})(\eta)|
\leq\Vinf{T^{-1}(\t{a_\la})-T^{-1}(\t{a_1})}\\
&=
\Vinf{\t{a_\la}-\t{a_1}}<\fr{5}{4},
\end{align*}
where we have used that $T$ is a surjective isometry on $(\SB, \Vinf{\cdot})$.
We arrive at a contradiction in this case, and hence $\phio(\la,\eta)=\phio(1,\eta)$
provided that $\Re\la\leq0$.
We next consider the case in which $\Re\la>0$, and then $\Re(-\la)<0$.
We derive from the fact proved above that $\phio(-\la,\eta)=\phio(1,\eta)$.
Since $\phio(-\la,\eta)=\phio(\la,\eta)$ by \eqref{lem2.6.1}, we conclude that
$\phio(\la,\eta)=\phio(1,\eta)$ even if $\Re\la>0$.
\end{proof}

By Lemma~\ref{lem3.2}, we may, and do, write $\phio(\la,\eta)=\phio(\eta)$
for each $\la\in\T$ and $\eta\in\MT$.
We believe this will cause no confusion.
Now, we shall describe the precise form of the mapping $\al(\la,\eta)$ on $\T$, witnessing that it is an isometry on $\T$ for each fixed $\eta\in\MT$.

\begin{lem}\label{lem3.3}
There exists a map $\e\colon\MT\to\set{\pm1}$ such that
\[
\al(\la,\eta)=\la^{\e(\eta)}\al(1,\eta)
\]
 for all $\la\in\T$ and $\eta\in\MT$.
\end{lem}

\begin{proof}
We first prove that
\begin{equation}\label{lem3.3.1}
|\la-\mu|\leq|\al(\la,\eta)\ov{\al(\mu,\eta)}-1|
\end{equation}
for all $\la,\mu\in\T$ and $\eta\in\MT$.
Fix arbitrary $\la,\mu\in\T$ and $\eta\in\MT$.
Setting $a_\nu=\al(\nu,\eta)e$ for $\nu=\la,\mu$,
we observe that $\t{a_\nu}\in\al(\nu,\eta)V_{\phi(\nu,\eta)}$, 
since $\te=1$ on $X\times\T$.
We thus obtain $\t{a_\nu}\in T(\nu V_\eta)$ by \eqref{lem3.1.1},
and in particular $T^{-1}(\t{a_\nu})(\eta)=\nu$ for $\nu=\la,\mu$.
Having in mind that $T$ is a surjective isometry and $\Vinf{\te}=1$,
we get
\begin{align*}
|\la-\mu|
&=
|T^{-1}(\t{a_\la})(\eta)-T^{-1}(\t{a_\mu})(\eta)|
\leq\Vinf{T^{-1}(\t{a_\la})-T^{-1}(\t{a_\mu})}\\
&=
\Vinf{\t{a_\la}-\t{a_\mu}}
=|\al(\la,\eta)-\al(\mu,\eta)|\,\Vinf{\te}
=|\al(\la,\eta)\ov{\al(\mu,\eta)}-1|,
\end{align*}
as claimed.

Taking $\mu=\pm1$ in \eqref{lem3.3.1}, we have
\[
|\la-1|\leq|\al(\la,\eta)\ov{\al(1,\eta)}-1|
\q\mbox{and}\q
|\la+1|\leq|\al(\la,\eta)\ov{\al(1,\eta)}+1|,
\]
where we have used \eqref{lem3.1.2}.
It follows from the last two inequalities that
\[
\al(\la,\eta)\ov{\al(1,\eta)}\in\set{\la,\ov{\la}},
\]
since $\al(\la,\eta)\in\T$;
in particular, if we consider the case $\la=i$, we get
$\al(i,\eta)\ov{\al(1,\eta)}\in\set{\pm i}$.
This assures the existence of $\e(\eta)\in\set{\pm1}$ such that
$\al(i,\eta)\ov{\al(1,\eta)}=\e(\eta)i$, and hence
$\al(i,\eta)=i\e(\eta)\al(1,\eta)$.
Since $\eta\in\MT$ is arbitrary, we have a well defined map $\e\colon\MT\to\set{\pm1}$.

Now we enter $\mu=i$ into \eqref{lem3.3.1} to obtain
\begin{align*}
|\la-i|
&\leq
|\al(\la,\eta)\ov{\al(i,\eta)}-1|
=|\al(\la,\eta)\,\ov{i\e(\eta)\al(1,\eta)}-1|\\
&=
|\al(\la,\eta)\e(\eta)\ov{\al(1,\eta)}-i|.
\end{align*}
Since $\al(-i,\eta)=-\al(i,\eta)$ by \eqref{lem3.1.2},
the same argument, applied to $\mu=-i$,
shows that $|\la+i|\leq|\al(\la,\eta)\e(\eta)\ov{\al(1,\eta)}+i|$.
These two inequalities ensure that
$\al(\la,\eta)\e(\eta)\ov{\al(1,\eta)}\in\set{\la,-\ov{\la}}$, and thus
\[
\al(\la,\eta)\ov{\al(1,\eta)}\in\set{\e(\eta)\la,-\e(\eta)\ov{\la}}.
\]
Therefore, we obtain
\[
\al(\la,\eta)\ov{\al(1,\eta)}\in\set{\la,\ov{\la}}\cap\set{\e(\eta)\la,-\e(\eta)\ov{\la}}.
\]
If $\la=\pm1$, then the equality $\al(\la,\eta)=\la^{\e(\eta)}\al(1,\eta)$ is true
by \eqref{lem3.1.2} with $\e(\eta)\in\set{\pm1}$.
We shall consider the case in which $\la\in\T\sm\set{\pm1}$.
First we consider the case when $\e(\eta)=1$, and hence
$\al(\la,\eta)\ov{\al(1,\eta)}\in\set{\la,\ov{\la}}\cap\set{\la,-\ov{\la}}$.
Since $\la\neq\pm1$, we must have $\al(\la,\eta)\ov{\al(1,\eta)}=\la$,
which shows that $\al(\la,\eta)=\la\al(1,\eta)$ provided that $\e(\eta)=1$.
We next consider the case when $\e(\eta)=-1$.
Then we have $\al(\la,\eta)\ov{\al(1,\eta)}\in\set{\la,\ov{\la}}\cap\set{-\la,\ov{\la}}$,
which implies $\al(\la,\eta)\ov{\al(1,\eta)}=\ov{\la}$, since $\la\neq\pm1$.
Consequently, we obtain $\al(\la,\eta)=\ov{\la}\al(1,\eta)$ provided that
$\e(\eta)=-1$.
We thus conclude that $\al(\la,\eta)=\la^{\e(\eta)}\al(1,\eta)$.
\end{proof}

For simplicity of notation, for each $\eta \in \MT$, we write $\al(\eta)$ instead of $\al(1, \eta)$. 
Applying Lemmata~\ref{lem3.2} and \ref{lem3.3},
equality \eqref{lem3.1.1} is rewritten as
\begin{equation}\label{lem3.3.2}
T(\la V_\eta)=\la^{\e(\eta)}\al(\eta)V_{(\phio(\eta),\phit(\la,\eta))}
\end{equation}
for all $\la\in\T$ and $\eta\in\MT$.

\begin{prop}\label{prop3.4}
If $\la\in\T$ and $(x,z)\in\M\times\T$, then
$\ov{\la}\der(a)(x)z=\Vinf{\der(a)}$
for all $\ta\in\la V_{(x,z)}$.
\end{prop}

\begin{proof}
First we note that $\la V_{(x,z)}$ is a maximal convex subset
of $\SB$ by Proposition~\ref{prop2.4}.
Let us fix $\ta\in\la V_{(x,z)}$.
By applying \eqref{tilde} and \eqref{V}, we get
$\con(a)+\der(a)(x)z=\ta (x,z)=\la$.
We deduce from the last equality and \eqref{norm} that
\begin{align*}
1
&=
\ov{\la}\con(a)+\ov{\la}\der(a)(x)z
\leq|\con(a)|+|\der(a)(x)|\\
&\leq|\con(a)|+\Vinf{\der(a)}
=\VB{a}=\Vinf{\ta}=1.
\end{align*}
If $\der(a)(x)=0$, then $\ov{\la}\con(a)=1$ and
$\Vinf{\der(a)}=0=\der(a)(x)$ by the last inequalities.
We shall consider the case in which $\der(a)(x)\neq0$.
The above inequalities show that
$|\ov{\la}\con(a)+\ov{\la}\der(a)(x)z|=|\con(a)|+|\der(a)(x)|$,
which guarantees the existence of $t\geq0$ with
$\ov{\la}\con(a)=t\ov{\la}\der(a)(x)z$.
It follows from the equality $\ov{\la}\con(a)+\ov{\la}\der(a)(x)z=1$ that
$\ov{\la}\der(a)(x)z=1/(1+t)$, and consequently, $\ov{\la}\con(a)=t/(1+t)$.
This implies, in particular, that $\ov{\la}\con(a)=|\con(a)|$.
We infer from the above inequalities that
$\ov{\la}\con(a)+\ov{\la}\der(a)(x)z=|\con(a)|+\Vinf{\der(a)}$, and therefore
we conclude $\ov{\la}\der(a)(x)z=\Vinf{\der(a)}$, as desired.
\end{proof}

Now we introduce the Hausdorff distance between maximal convex subsets in $\SB$, which is used in \cite{hat}, \cite{hat3} and \cite{mor} to solve Tingley's problem in some concrete cases.

\begin{defn}\label{defn1}
	Let $F_1,F_2$ be two maximal convex subsets of $\SB$.
	The Hausdorff distance between $F_1$ and $F_2$, $\dh(F_1,F_2)$, is defined by
	\begin{equation}\label{defn1.1}
		\dh(F_1,F_2)=
		\max\left\{\sup_{\ta\in F_1}d(\ta,F_2),
		\sup_{\tb\in F_2}d(F_1,\tb)\right\},
	\end{equation}
	where $d(\ta,F_2)=\inf_{\t{c}\in F_2}\Vinf{\ta-\t{c}}$ and
	$d(F_1,\tb)=\inf_{\t{c}\in F_1}\Vinf{\t{c}-\tb}$.
\end{defn}

Here we emphasize that the surjective isometry $T$ on $\SB$ preserves maximal convex subsets (cf. \cite[Lemma~3.5]{tana}) and the Hausdorff distance between any two maximal convex subsets of $\SB$.
To be more explicit, we have
\begin{equation}\label{defn1.2}
	\dh(T(F_1),T(F_2))=\dh(F_1,F_2)
\end{equation}
for all maximal convex subsets $F_1,F_2$ of $\SB$.

In our next lemma we compute the Hausdorff distance between
particular maximal convex subsets
of $\SB$.

\begin{lem}\label{lem3.5}
Let $k\in\set{\pm1}$, $x\in\cha$ and $z\in\T$.
Then we have
\[
\sup_{\ta\in kV_{(x,k)}}d(\ta,kV_{(x,z)})
=\sup_{\tb\in kV_{(x,z)}}d(kV_{(x,k)},\tb)=|k-z|.
\]
In particular, we have $\dh(kV_{(x,k)},kV_{(x,z)})=|k-z|$.
\end{lem}

\begin{proof}
Let us fix $\ta\in kV_{(x,k)}$ and $\tb\in kV_{(x,z)}$.
By \eqref{tilde} and \eqref{V}, we get
\begin{equation}\label{lem2.11.2}
\con(a)+\der(a)(x)k=\ta (x,k)=k
\q\mbox{and}\q
\con(b)+\der(b)(x)z= \tb (x,k)=k.
\end{equation}
We have $\der(a)(x)=\Vinf{\der(a)}$ and
$k\der(b)(x)z=\Vinf{\der(b)}$ by Proposition~\ref{prop3.4}.
We derive from \eqref{norm} and \eqref{lem2.11.2} that
\begin{align*}
|(k-z)\der(a)(x)|
&\leq
|k\der(a)(x)-\der(b)(x)z|+|\der(b)(x)z-z\der(a)(x)|\\
&=
|(k-\con(a))-(k-\con(b))|+|\der(b)(x)-\der(a)(x)|\\
&\leq
|\con(b)-\con(a)|+\Vinf{\der(b)-\der(a)}
=\VB{b-a}=\Vinf{\ta-\tb}.
\end{align*}
Since $\der(a)(x)=\Vinf{\der(a)}$, we obtain
$|k-z|\,\Vinf{\der(a)}\leq\Vinf{\ta-\tb}$.
By the same reasonings, we get
\begin{align*}
|(k-z)\der(b)(x)|
&\leq
|k\der(b)(x)-\der(a)(x)k|+|\der(a)(x)k-z\der(b)(x)|\\
&=
|\der(b)(x)-\der(a)(x)|+|(k-\con(a))-(k-\con(b))|\\
&\leq
\Vinf{\der(b)-\der(a)}+|\con(b)-\con(a)|
=\VB{b-a}=\Vinf{\ta-\tb}.
\end{align*}
Noting that $k\der(b)(x)z=\Vinf{\der(b)}$, we have
$|k-z|\,\Vinf{\der(b)}\leq\Vinf{\ta-\tb}$.
Since $\ta\in kV_{(x,k)}$ and $\tb\in kV_{(x,z)}$ are arbitrary,
we conclude that 
\begin{equation}\label{eq bounds from below in L3.5} |k-z|\,\Vinf{\der(a)}\leq d(\ta,kV_{(x,z)}),
	\hbox{ and }
	|k-z|\,\Vinf{\der(b)}\leq d(kV_{(x,k)},\tb),
\end{equation} for all $\ta\in kV_{(x,k)}$ and $\tb\in kV_{(x,z)}$.

Now we set $b_0=\con(a)e+k\ov{z}\,\mathcal{I}(\der(a))$
and $a_0=\con(b)e+kz\mathcal{I}(\der(b))$.
Since $\con$ and $\der$ are both linear maps satisfying hypotheses $(\ref{(1)})$ and $(\ref{(4)})$,
$\con(e)=1$ and $\der(e)=0$, we get $\t{b_0}(y,\nu)=\con(a)+k\ov{z}\,\der(a)(y)\nu$
and $\t{a_0}(y,\nu)=\con(b)+kz\der(b)(y)\nu$ for all $(y,\nu)\in X\times\T$.
In particular, $\t{b_0}(x,z)=k$ and $\t{a_0}(x,k)=k$ by \eqref{lem2.11.2},
which yields $\t{b_0}\in kV_{(x,z)}$ and $\t{a_0}\in kV_{(x,k)}$.
Having in mind that $\ta(y,\nu)=\con(a)+\der(a)(y)\nu$ and
$\tb(y,\nu)=\con(b)+\der(b)(y)\nu$ by \eqref{tilde}, we obtain
\begin{align*}
\Vinf{\ta-\t{b_0}}
&=
\sup_{(y,\nu)\in X\times\T}|\tf(y,\nu)-\t{b_0}(y,\nu)|
=\sup_{(y,\nu)\in X\times\T}|(1-k\ov{z})\der(a)(y)|\\
&=
|1-k\ov{z}|\,\Vinf{\der(a)}
=|k-z|\,\Vinf{\der(a)},
\end{align*}
and 
\begin{align*}
	\Vinf{\t{a_0}-\tb}
	&=
	\sup_{(y,\nu)\in X\times\T}|(kz-1)\der(b)(y)|
	=|k-z|\,\Vinf{\der(b)}.
\end{align*} Since $\t{b_0}\in kV_{(x,z)}$ and $\t{a_0}\in kV_{(x,k)}$, by \eqref{eq bounds from below in L3.5}
we have $$d(\ta,kV_{(x,z)})= |k-z|\,\Vinf{\der(a)} \le |k - z|$$ and
$$d(kV_{(x,k)},\tb)= |k-z|\,\Vinf{\der(b)} \le |k - z|.$$
By the arbitrariness of $\ta $ and $\tb$ we arrive to \begin{equation}\label{eq bound from above in L3.5} \sup_{\ta\in kV_{(x,k)}}d(\ta,kV_{(x,z)}),\sup_{\tb\in kV_{(x,z)}}d(kV_{(x,k)},\tb)\leq |k-z|.
\end{equation}
Since $A$ is extremely C-regular, there exists $f_1\in S_{A}$ satisfying $f_1(x)=1$ (and an additional property not required here). Define $a_1=\INT(f_1)$ and $b_1=\INT(k \ov{z} f_1)$ in $B$.  It is not hard to see that $\t{a_1}\in kV_{(x,k)}$, $\t{b_1}\in kV_{(x,z)}$ with $\|a_1\|_{B}= \Vinf{\der(a_1)}=1=\Vinf{\der(b_1)} = \|b_1\|_{B}$ by hypotheses $(\ref{(1)}),$ $(\ref{(4)})$ and \eqref{tilde}.
	Therefore, by \eqref{eq bounds from below in L3.5} and \eqref{eq bound from above in L3.5} we obtain
	$$d(\t{a_1},kV_{(x,z)}), d(kV_{(x,k)},\t{b_1})= |k-z|.$$ By a new application of \eqref{eq bound from above in L3.5}, we conclude that  
	$$\sup_{\ta\in kV_{(x,k)}}d(\ta,kV_{(x,z)})=|k-z|
	=\sup_{\tb\in kV_{(x,z)}}d(kV_{(x,k)},\tb).\eqno\qedhere$$
\end{proof}

\begin{rem}\label{rem1} In \cite[Proposition~4.4.]{hat}, Hatori proved that each complex Banach space satisfying the Hausdorff distance condition besides another technical assumption has the complex Mazur--Ulam property. We shall show that each complex Banach space $B$ satisfying the hypotheses in Theorem~\ref{thm2}  fails the Hausdorff distance condition.  For this purpose we refresh first some notation.\smallskip

Let $\mathcal{F}$ be the set of all maximal convex subsets of $\SAo$ and $\mathcal{Q}=\set{p\in\ext{B^*_1}:p^{-1}(1)\cap\SAo\in\mathcal{F}}$. For each pair of $p\in\mathcal{Q}$ and $\la\in\T$, we denote
$F_{p,\la}=\set{a\in\SAo:p(a)=\la}$,
that is, $F_{p,\la}=(\ov{\la}p)^{-1}(1)\cap\SAo$.
We define a binary relation $\sim$ on $\mathcal{Q}$:
For $p,q\in\mathcal{Q}$, $p\sim q$ if there exists $\la\in\T$ such that
$F_{p,1}=F_{q,\la}$.
We can check that $\sim$ is an equivalence relation on $\mathcal{Q}$.
A set of representatives for the quotient space $\mathcal{Q}/{\sim}$
is a subset  $\mathcal{P}$ of $\mathcal{Q}$ containing one and only
one element in each class of the quotient space.
Under this notation, for each $K\in\mathcal{F}$ there exists a unique
pair of $(p,\la)\in\mathcal{P}\times\T$ such that $K=F_{p,\la}$
(see \cite[Lemma~2.5]{hat}).
According to \cite[Definition~3.2]{hat}, we say that $B$ satisfies
the {\it Hausdorff distance condition} if $\dh(F_{p,\la},F_{q,\mu})=2$
for every $(p,\la),(q,\mu)\in\mathcal{P}\times\T$ with $p\neq q$
in $\mathcal{P}$.

We are now ready to show that each subspace $B$ fulfilling the hypotheses of Theorem~\ref{thm2} fails the Hausdorff distance condition.
Take $x\in\cha$, and set $\eta_1=(x,1),$ $\eta_2=(x,i)\in\MT$.
Then $V_{\eta_1}$ and $V_{\eta_2}$ are maximal convex subsets
of $\SB$ by Proposition~\ref{prop2.4}.
We have $\dh(V_{\eta_1},V_{\eta_2})=|1-i|=\sqrt{2}$ by Lemma~\ref{lem3.5}.
Let $U$ be the complex linear isometry from $\Ao$ onto $\TB$ considered in the introduction, that is, $U(a)=\ta$ for $a\in B$.
By \cite[Lemma~3.5]{tana}, we see that $U^{-1}(V_{\eta_j})$ is a maximal
convex subset of $\SAo$ for $j=1,2$.
There exists $(p_j,\la_j)\in\mathcal{P}\times\T$ such that
$U^{-1}(V_{\eta_j})=F_{p_j,\la_j}$.
Since $U\colon\Ao\to\TB$ is a surjective isometry,
we observe that $U$ preserves the Hausdorff distance,
that is, $\dh(F_{p_1,\la_1},F_{p_2,\la_2})=\dh(V_{\eta_1},V_{\eta_2})=\sqrt{2}$
(see \eqref{defn1.2}).
We shall prove that $p_1\neq p_2$ in $\mathcal{P}$.
Suppose, on the contrary, that $p_1=p_2$ in $\mathcal{P}$,
that is, $p_1\sim p_2$ in $\mathcal{Q}$. 
Then there exists $\la_0\in\T$ such that
$F_{p_1,1}=F_{p_2,\la_0}$.
Note that $F_{p,\la}=\la F_{p,1}$ for all $(p,\la)\in\mathcal{P}\times\T$; in fact,
\begin{align*}
F_{p,\la}
&=
\set{a\in\SAo:p(a)=\la}
=\set{a\in\SAo:p(\ov{\la}a)=1}\\
&=\la\set{b\in\SAo:p(b)=1}=\la F_{p,1}.
\end{align*}
Having in mind that $U^{-1}(V_{\eta_j})=F_{p_j,\la_j}$, we get
\begin{align*}
V_{\eta_1}
&=
U(F_{p_1,\la_1})=U(\la_1F_{p_1,1})=\la_1U(F_{p_1,1})
=\la_1U(F_{p_2,\la_0})\\
&=\la_1\la_0\ov{\la_2}U(F_{p_2,\la_2})
=\la_1\la_0\ov{\la_2}V_{\eta_2}.
\end{align*}
Applying Lemma~\ref{lem2.1}, we have $\eta_1=\eta_2$,
which contradicts $\eta_1=(x,1)$ and $\eta_2=(x,i)$.
Consequently, $p_1\neq p_2$ in $\mathcal{P}$.
Therefore, $(p_1,\la_1),(p_2,\la_2)\in\mathcal{P}\times\T$ satisfy
$p_1\neq p_2$ in $\mathcal{P}$ and
$\dh(F_{p_1,\la_1},F_{p_2,\la_2})=\sqrt{2}\neq2$.
This shows that $B$ fails the Hausdorff distance condition.
\end{rem}

In the next three results, we determine the behaviour of the maps
$\phio$ and $\phit$.

\begin{lem}\label{lem3.6}
For each $x\in\M$ and $z\in\T$, we have $\phio(x,z)=\phio(x,1)$.
\end{lem}

\begin{proof} Fix arbitrary $k\in\set{\pm1},$ $x\in\M$ and $z\in\T\sm\set{\pm1}$.
Suppose, contrary to the desired statement, that $\phio(x,z)\neq\phio(x,k)$.
By \cite[Proposition~5.4 and Example~5.2]{hat3} and having in mind that $A$ is extremely C-regular, we can find $f_k\in\SA$ such that
\begin{align*}
f_k(\phio(x,z))
&=
k \al(x,z) \ov{\phit(k,(x,z))}
\q\mbox{and}\\
f_k(\phio(x,k))
&=
-k \al(x,k) \ov{\phit(k,(x,k))}.
\end{align*}
Here we set $b_k=\mathcal{I}(f_k)$.
Since $\phio(\la,\eta)=\phio(\eta)$ by Lemma~\ref{lem3.2}, we see that
\[
\t{b_k}(\phi(k,(x,z)))=k\al(x,z)
\q\mbox{and}\q
\t{b_k}(\phi(k,(x,k)))=-k\al(x,k)
\]
by \eqref{tilde}, and hypotheses $(\ref{(1)})$ and $(\ref{(4)})$.
This shows that
\[
\t{b_k}\in k\al(x,z)V_{\phi(k,(x,z))}\cap(-k\al(x,k)V_{\phi(k,(x,k))}).
\]
It follows from \cite[Lemma~3.1]{hat} that
\[
\dh(k\al(x,z)V_{\phi(k,(x,z))},k\al(x,k)V_{\phi(k,(x,k))})=2.
\]
Here, we note that
$k\al(x,z)V_{\phi(k,(x,z))}=T(kV_{(x,z)})$ and
$k\al(x,k)V_{\phi(k,(x,k))}=T(kV_{(x,k)})$ by \eqref{lem3.1.1}
with Lemma~\ref{lem3.3}, where $T$ is the surjective isometry on $S_{\tilde{B}}$ given by $T=U\del U^{-1}$.
Thus,
\begin{align*}
2
&=
\dh(k\al(x,z)V_{\phi(k,(x,z))},k\al(x,k)V_{\phi(k,(x,k))})\\
&=
\dh(T(kV_{(x,z)}),T(kV_{(x,k)}))
=\dh(kV_{(x,z)},kV_{(x,k)})=|k-z|
\end{align*}
by Lemma~\ref{lem3.5}.
Hence $|k-z|=2$, and then
we must have $z=-k$, which contradicts $z\neq\pm1$.
This shows that $\phio(x,z)=\phio(x,k)$ provided $z\neq\pm1$.
In particular, $\phio(x,i)=\phio(x,k)$ for $k=\pm1$, which yields
$\phio(x,-1)=\phio(x,i)=\phio(x,1)$.
This proves that $\phio(x,z)=\phio(x,1)$ for all $x\in\cha$ and $z\in\T$.
\end{proof}

The next result plays a fundamental role to determine the map $\phit$.

\begin{lem}\label{lem3.7}
The inequalities 
\begin{equation}\label{lem3.7.1}
\begin{aligned}
|\la^{\e(\eta)}\ov{\phit(\la,\eta)}\phit(\mu,\eta)-\mu^{\e(\eta)}|
&\leq
|\la-\mu|,
\q\mbox{and}\\
|\la^{\e(\eta)}\ov{\phit(\la,\eta)}\phit(\mu,\eta)+\mu^{\e(\eta)}|
&\leq|
\la+\mu|,	
\end{aligned}
\end{equation}
hold for all $\la,\mu\in\T$ and $\eta\in\M\times\T$.
\end{lem}

\begin{proof}
Fix arbitrary $\la,\mu\in\T$ and $\eta\in\M\times\T$.
We derive from \cite[Lemma~3.1]{hat} that
\begin{equation}\label{lem3.7.2}
\dh(\la V_\eta,\mu V_\eta)=|\la-\mu|.
\end{equation}
Since $\phio(\eta)=\phio(\la,\eta)\in\cha$, having in mind that $A$ is extremely C-regular, we can find $f_0\in S_{A}$ satisfying $f_0 (\phio(\eta)) =1$ (and an additional property not required here). By defining $a_1=\al(\la,\eta)\ov{\phit(\la,\eta)}\mathcal{I}(f_0)\in S_{B}$,  we see that 
$\con(a_1)=0$ and $\der(a_1)(\phio(\eta))\phit(\la,\eta)=\al(\la,\eta)$
by hypotheses $(\ref{(1)})$ and $(\ref{(4)})$.
In particular, we have $\t{a_1}\in\al(\la,\eta)V_{\phi(\la,\eta)}=T(\la V_\eta)$
by \eqref{tilde}
and \eqref{lem3.1.1}.
Setting $\t{b_1}=T(\tg)$ for each $\tb\in\mu V_\eta$, we get
$\t{b_1}\in T(\mu V_\eta)=\al(\mu,\eta)V_{\phi(\mu,\eta)}$.
By definition, we obtain
$\con(b_1)+\der(b_1)(\phio(\eta))\phit(\mu,\eta)=\al(\mu,\eta)$.
Since $\con(a_1)=0$, we have
\begin{multline*}
|\al(\la,\eta)\ov{\phit(\la,\eta)}-\al(\mu,\eta)\ov{\phit(\mu,\eta)}|\\
\leq
|\der(a_1)(\phio(\eta))-\der(b_1)(\phio(\eta))|+|\con(b_1)\ov{\phit(\mu,\eta)}|\\
\leq
|\con(a_1)-\con(b_1)|+\Vinf{\der(a_1)-\der(b_1)}
=\VB{a_1-b_1}=\Vinf{\t{a_1}-\t{b_1}}.
\end{multline*}
Noting that, by Lemma~\ref{lem3.3}, $\al(\nu,\zeta)=\nu^{\e(\zeta)}\al(\zeta)$ for $\nu\in\T$ and
$\zeta\in\MT$, we deduce from the above
inequalities that
$|\la^{\e(\eta)}\ov{\phit(\la,\eta)}-\mu^{\e(\eta)}\ov{\phit(\mu,\eta)}|
\leq\Vinf{\t{a_1}-\t{b_1}}$.
Since $\t{b_1}=T(\tb)\in T(\mu V_\eta)$ is arbitrary, we get
$$|\la^{\e(\eta)}\ov{\phit(\la,\eta)}-\mu^{\e(\eta)}\ov{\phit(\mu,\eta)}|
\leq d(\t{a_1},T(\mu V_\eta)).$$
We infer from \eqref{defn1.1}, \eqref{defn1.2} and \eqref{lem3.7.2} that
\begin{multline*}
|\la^{\e(\eta)}\ov{\phit(\la,\eta)}-\mu^{\e(\eta)}\ov{\phit(\mu,\eta)}|
\leq\sup_{T(\tf)\in T(\la V_\eta)}d(T(\tf),T(\mu V_\eta))\\
\leq\dh(T(\la V_\eta),T(\mu V_\eta))=\dh(\la V_\eta,\mu V_\eta)
=|\la-\mu|.
\end{multline*}
Thus, $|\la^{\e(\eta)}\ov{\phit(\la,\eta)}\phit(\mu,\eta)-\mu^{\e(\eta)}|\leq|\la-\mu|$.
Since $\phit(-\mu,\eta)=\phit(\mu,\eta)$ by \eqref{lem2.6.1}, we obtain
$|\la^{\e(\eta)}\ov{\phit(\la,\eta)}\phit(\mu,\eta)+\mu^{\e(\eta)}|\leq|\la+\mu|$.
\end{proof}

\begin{lem}\label{lem3.8}
There exists a map $\et\colon\MT\to\set{\pm1}$
satisfying $\phit(\la,\eta)=\la^{\e(\eta)-\et(\eta)}\phit(1,\eta)$ for all $\la\in\T$
and $\eta\in\MT$.
\end{lem}

\begin{proof}
Let us fix $\eta\in\MT$ and $\la\in\T\setminus\set{\pm1}$ arbitrarily.
According to \eqref{lem3.7.1} with $\mu=1$, we get
$|\la^{\e(\eta)}\ov{\phit(\la,\eta)}\phit(1,\eta)\pm1|\leq|\la\pm1|$.
These two inequalities imply
$\la^{\e(\eta)}\ov{\phit(\la,\eta)}\phit(1,\eta)\in\set{\la,\ov{\la}}$,
that is,
\[
\ov{\phit(\la,\eta)}\phit(1,\eta)\in\set{\la^{1-\e(\eta)},\la^{-1-\e(\eta)}}.
\]
Entering $\la=i$ into the above to obtain
$\ov{\phit(i,\eta)}\phit(1,\eta)\in\set{\pm\e(\eta)}$.
Then there exists $\et(\eta)\in\set{\pm1}$ such that
$\ov{\phit(i,\eta)}\phit(1,\eta)=\et(\eta)\e(\eta)$.
Since $\eta\in\MT$ is arbitrary, we may, and do, regard $\et$ as a map
from $\MT$ to $\set{\pm1}$.
Considering the case in which $\mu=i$ in \eqref{lem3.7.1}, we get 
$$\begin{aligned}
	|\la^{\e(\eta)}\ov{\phit(\la,\eta)}\phit(i,\eta)-\e(\eta)i|& \leq|\la-i|,
	\q\mbox{and} \\
	|\la^{\e(\eta)}\ov{\phit(\la,\eta)}\phit(i,\eta)+\e(\eta)i|&\leq|\la+i|.
\end{aligned}
$$
By substituting $\phit(i,\eta)$ with $\et(\eta)\e(\eta)\phit(1,\eta)$ in the above
two inequalities, we have 
$$\begin{aligned}
	|\la^{\e(\eta)}\ov{\phit(\la,\eta)}\et(\eta)\phit(1,\eta)-i|& \leq|\la-i|,
	\q\mbox{and}\\
	|\la^{\e(\eta)}\ov{\phit(\la,\eta)}\et(\eta)\phit(1,\eta)+i|& \leq|\la+i|,
\end{aligned}
$$ which combined imply that
$\la^{\e(\eta)}\ov{\phit(\la,\eta)}\et(\eta)\phit(1,\eta)\in\set{\la,-\ov{\la}}$,
leading us to
\[
\ov{\phit(\la,\eta)}\phit(1,\eta)\in
\set{\et(\eta)\la^{1-\e(\eta)},-\et(\eta)\la^{-1-\e(\eta)}}
\cap
\set{\la^{1-\e(\eta)},\la^{-1-\e(\eta)}}.
\]
We have two possible cases to consider.
If $\et(\eta)=1$, then
\[
\ov{\phit(\la,\eta)}\phit(1,\eta)\in
\set{\la^{1-\e(\eta)},\la^{-1-\e(\eta)}}\cap\set{\la^{1-\e(\eta)},-\la^{-1-\e(\eta)}}.
\]
Since $\la\neq\pm1$, we must have 
$\ov{\phit(\la,\eta)}\phit(1,\eta)=\la^{1-\e(\eta)}$.
If $\et(\eta)=-1$, then $\ov{\phit(\la,\eta)}\phit(1,\eta)=\la^{-1-\e(\eta)}$
by a quite similar argument as above.
These conclusions prove that
$\ov{\phit(\la,\eta)}\phit(1,\eta)=\la^{\et(\eta)-\e(\eta)}$
for $\la\in\T\setminus\set{\pm1}$. The last equality holds even for $\la\in\set{\pm1}$ by \eqref{lem2.6.1} and what we have just proved. Thus, the map $\et$ satisfies $\phit(\la,\eta)=\la^{\e(\eta)-\et(\eta)}\phit(1,\eta)$
for all $\la\in\T$ and $\eta\in\MT$.
\end{proof}

\section{Characterizations of $\con\circ\del$ and $\der\circ\del$}

Let $\la\in\T$ and $\eta\in\MT$.
By \eqref{lem3.1.1}, $T(\la V_\eta)=\al(\la,\eta)V_{\phi(\la,\eta)}$, and thus
$T(\tf)(\phi(\la,\eta))=\al(\la,\eta)$ for $\ta\in\la V_\eta$. By applying Lemma~\ref{lem3.3} with the notation of $\al(\eta)=\al(1,\eta)$,
we have
\begin{equation}\label{(4.1)}
\al(\la,\eta)=\la^{\e(\eta)}\al(\eta)
\end{equation}
for $\la\in\T$ and $\eta\in\MT$.
We shall write $\phio(\la,(x,z))=\phio(x)$ for $\la\in\T$ and $(x,z)\in\MT$,
which is appropriate by Lemmata~\ref{lem3.2} and \ref{lem3.6}.
Then $\phio$ is a surjective map from $\cha$ onto itself.
Since $T(\tf)=\t{\del(a)}$ by \eqref{T}, we can rewrite 
the identity $T(\ta)(\phi(\la,\eta))=\al(\la,\eta)$ as
\begin{equation}\label{lem4.1.1}
\con(\del(a))+\der(\del(a))(\phio(x))\phit(\la,\eta)=\al(\la,\eta)
\end{equation}
for all $\la\in\T$, $\eta=(x,z)\in\M\times\T$ and $a\in\SAo$ with $\ta\in\la V_\eta$.
We shall also write $\phit(1,\eta)=\phit(\eta)$.
By Lemma~\ref{lem3.8}, we get
\begin{equation}\label{(4.3)}
\phit(\la,\eta)=\la^{\e(\eta)-\et(\eta)}\phit(\eta)
\end{equation}
for each $\la\in\T$ and $\eta\in\MT$.
If we apply \eqref{(4.1)} and \eqref{(4.3)}
to \eqref{lem4.1.1}, then we obtain
\begin{equation}\label{lem4.1.2}
\con(\del(a))+\der(\del(a))(\phio(x))\la^{\e(\eta)-\et(\eta)}\phit(\eta)
=\la^{\e(\eta)}\al(\eta)
\end{equation}
for $\la\in\T$, $\eta=(x,z)\in\M\times\T$ and $a\in\SAo$ satisfying
$\ta\in\la V_\eta$.
Under the above notations, we can rewrite \eqref{lem3.3.2} as
\begin{equation}\label{lem4.1.3}
T(V_\eta)=\al(\eta)V_{(\phio(x),\phit(\eta))}
\end{equation}
for all $\eta=(x,z)\in\MT$.

\begin{lem}\label{lem4.1}
Let $\la_0\in\T$, $x\in \cha$. Suppose that $\con(\del(\la_0\unitI))=0$. Then $\der(\del(\la_0\INT(f))) (\phi_1(x))=0,$ for every $f\in S_{A}$ with $|f(x)| =1$.
\end{lem}

\begin{proof} Set $a_0=\la_0\INT(f)$, where $f$ is a function satisfying the hypothesis of the lemma. By hypotheses $(\ref{(1)})$ and $(\ref{(4)})$ in Theorem~\ref{thm2}, we have $\der(a_0)=\la_0 f$ and $\con(a_0)=0$. Let us denote $\mu=f(x)$. Setting $\eta_z=(x,z)$ for each $z\in\T$, we observe that $\t{a_0}\in\la_0 \mu z V_{\eta_z}$
by \eqref{tilde}.
It follows from \eqref{lem4.1.2} that
\begin{equation}\label{lem4.1.4}
\begin{aligned}
	\con(\del(a_0))&+\der(\del(a_0))
	(\phio(x))(\la_0\mu z)^{\e(\eta_z)-\et(\eta_z)}\phit(\eta_z)\\
	&= \widetilde{\del(a_0)} ( \phi (\la_0 \mu z,\eta_z)) =(\la_0\mu z)^{\e(\eta_z)}\al(\eta_z).
\end{aligned}
\end{equation}
The identity $\te=1$ on $X\times \mathbb{T}$ implies that $\la_0\te\in\la_0V_{\eta_z}$. It follows from our assumption that $\con(\del(\la_0 e))=0,$ and hence
\eqref{lem4.1.2} with $\la_0\t{e}\in\la_0V_{\eta_z}$
assure us that
\[
\der(\del(\la_0\unitI))(\phio(x))\la_0^{-\et(\eta_z)}\phit(\eta_z)=\al(\eta_z).
\]
Set $h(x)=\der(\del(\la_0\unitI))(\phio(x))$, and then the last equality writes as
\[
h(x)\la_0^{-\et(\eta_z)}\phit(\eta_z)=\al(\eta_z).
\]
Applying the above equality to \eqref{lem4.1.4}, we obtain
$$\begin{aligned}
 &\left[(\la_0\mu z)^{\e(\eta_z)}h(x)\la_0^{-\et(\eta_z)}
-\der(\del(a_0))(\phio(x))(\la_0\mu z)^{\e(\eta_z)-\et(\eta_z)}\right]\phit(\eta_z) \\
&= \con(\del(a_0)) =(\la_0\mu z)^{\e(\eta_z)-\et(\eta_z)}
\left[h(x)(\mu z)^{\et(\eta_z)}
-\der(\del(a_0))(\phio(x))\right]\phit(\eta_z).
\end{aligned}$$
We get $$|\con(\del(a_0))|=|(\mu z)^{\et(\eta_z)}
-\ov{h(x)}\,\der(\del( a_0))(\phio(x))|,$$ since $|h(x)|=|\la_0^{\et(\eta_z)}\ov{\phit(\eta_z)}\al(\eta_z)|=1$.
Noting that $z\in\T$ was arbitrarily chosen, we obtain
$$|\con(\del(a_0))|=|\nu-\ov{h(x)}\,\der(\del(a_0))(\phio(x))|$$ for all   $\nu\in\set{\pm1, i^{\et(\eta_{i \ov{\mu}})} }$. We derive from $\et(\eta_{i \ov{\mu}})\in\set{\pm1}$ that
$$\ov{h(x)}\,\der(\del(a_0))(\phio(x))=0, \hbox{ equivalently, } \der(\del( a_0))(\phio(x))=0.\eqno\qedhere$$
\end{proof}

Our next result also is an important tool in our argument.
But in the statement we need an additional
assumption for $\cha$ (see Remark~\ref{rem2} below).

\begin{lem}\label{lem4.2}
If $\cha$ contains at least two points,
then $\con(\del(\la e))$ is a non-zero scalar for all $\la\in\T$.
\end{lem}

\begin{proof}
Suppose, on the contrary, that $\con(\del(\la_0e))=0$ for some $\la_0\in\T$. Take $x_1,x_2\in\cha$ with $x_1\neq x_2$; this is possible by assumption. Since $A$ is extremely C-regular, there exists $f_0\in\SA$ such that $f_0(x_1)=1 =f_0(x_2)$ (cf. \cite[Proposition~5.4]{hat3}). Set $a_0=\la_0\INT({f_0})$. Lemma~\ref{lem4.1} assures that $\der(\del(a_0))\left(\phi_1(x_j)\right)=0$ for all $j=1,2$. By the choice of $a_0$, we see that $\t{a_0}\in\la_0 z V_{(x_1,z)}$ for all $z\in\T$. Since $\der(\del(a_0))\left(\phi_1(x_j)\right)=0$, we get
	\begin{equation}\label{lem4.2.0new}
		\con(\del(a_0))=\al(\la_0 z,(x_1,z))
	\end{equation} for all $z\in\T$ by \eqref{lem4.1.1}. In particular, 
\begin{equation}\label{lem4.2.1}
\al(\la_0,(x_1,1))=\al(-\la_0,(x_1,-1)), \hbox{ and } |\con(\del(a_0))|=1.
 \end{equation} Let us observe that $$1 = \|\del(a_0)\|_{B} = |\con(\del(a_0))| + \|\der(\del(a_0))\|_{\infty},$$ which proves that $\|\der(\del(a_0))\|_{\infty} =0$.

By applying once again that $A$ is extremely C-regular, there exists $f_1\in\SA$ such that
{$f_1(x_1)=1$ and $f_1(x_2)=-1$} (cf. \cite[Proposition~5.4]{hat3}). We define $f_2=(f_1+{f_0})/2\in \SA$. Then we observe that $f_2(x_1)=1$ and $\Vinf{f_2 - {f_0}}=1$.
Setting $a_1=(a_0+\la_0\INT(f_{2}))/2=(\la_0\INT(f_0)+\la_0\INT(f_2))/2$, we obtain $\con(a_1)=\la_0\con(\INT(f_0 + f_2 ) ) /2=0$ and
$\der(a_1)=\la_0(f_0+f_2)/2$ by hypotheses $(\ref{(1)})$ and $(\ref{(4)})$, and thus $\der(a_1) (x_1) = \la_0$.
This shows that $\t{a_1}\in\la_0 z V_{(x_1,z)}$ for all $z\in\T$. We therefore deduce from \eqref{lem4.1.1}
and \eqref{lem4.2.0new} that
\begin{equation}\label{lem4.2.2}
\con(\del(a_1))+\der(\del(a_1)) (\phio(x_1)) \phit(\la_0 z,(x_1,z)) = \con(\del(a_0))
\end{equation} for all $z\in\T$.
We thus obtain
$$|\con(\del(a_1))-\con(\del(a_0))|
=|\der(\del(a_1))(\phio(x_1))|\leq\Vinf{\der(\del(a_1))},$$
which leads to
\begin{equation}\label{eq norm of derDelta a_1 greater 1/2}
	\begin{aligned}
		2\Vinf{\der(\del(a_1))}
		&\geq
		|\con(\del(a_1))-\con(\del(a_0))|+\Vinf{\der(\del(a_1))}\\
		&=
		|\con(\del(a_1))-\con(\del(a_0))|
		+\Vinf{\der(\del(a_1))-\der(\del(a_0))}\\
		&=
		\VB{\del(a_1)-\del(a_0)}=\VB{a_1-a_0}
		=\fr{1}{2}\VB{\la_0 \INT(f_2- f_0)}\\
		&=
		\fr{1}{2}\left(|\con(\la_0\INT(f_2-f_0))|
		+\Vinf{\der(\la_0\INT(f_2-f_0))}\right)\\
		&=
		\fr{1}{2}\,
		\Vinf{\la_0(f_2-f_0)}
		=
		\fr{1}{2},
	\end{aligned}
\end{equation}
where we have used hypotheses $(\ref{(1)})$ and $(\ref{(4)})$ in Theorem~\ref{thm1}. However, since $(f_0+f_2)/2 (x_1) = 1$, Lemma~\ref{lem4.1} implies that $\der(\del(a_1))\left(\phi_1(x_1)\right)=\der(\del(\la_0\INT(f_0+f_2))/2))\left(\phi_1(x_1)\right)=0$. Furthermore, $\t{a_1}\in\la_0 z V_{(x_1,z)}$ for all $z\in\T$ and $$ \con(\del(a_1))=\al(\la_0 z,(x_1,z))\in \mathbb{T} \hbox{ for all $z\in\T$ by \eqref{lem4.1.1}}$$ and the equality $$1 = \|\del(a_1)\|_{B} = |\con(\del(a_{1}))| + \|\der(\del(a_{1}))\|_{\infty},$$ leads to $\|\der(\del(a_{1}))\|_{\infty} =0,$ which contradicts \eqref{eq norm of derDelta a_1 greater 1/2}. 
\end{proof}

\begin{rem}\label{rem2}
We have not excluded the case in which $\cha$ is a one point set until now.
This extra hypothesis is necessary, since as we shall show next,
Lemma~\ref{lem4.2} need not be true if $\cha
=\set{x_0}$ is a single point.
As a matter of fact, if $\cha=\set{x_0}$, then the point evaluation $\d_{x_0}\colon A\to\C$
is a surjective complex linear isometry.
Thus, we may, and do, regard $A\equiv \C$, and then
$\der$ is a bounded linear functional on $\Ao$. We shall write $1$ for the unit in $\mathbb{C} \equiv A$.
Fix an arbitrary $a\in\Ao$ and set $a_0=\con(a)e+\der(a)\INT({1})\in\Ao$.
Since $\con$ and $\der$ are linear maps satisfying hypotheses $(\ref{(1)})$, $(\ref{(4)})$
and $(\ref{(6)})$ in Theorem \ref{thm1}, we get
$\con(a_0)=\con(a)$ and $\der(a_0)=\der(a)$.
We thus obtain
\[
\V{a_0-a}_B=|\con(a_0)-\con(a)|+|\der(a_0)-\der(a)|=0,
\]
and hence $a_0=a$.
This shows that $a=\con(a)e+\der(a)\INT({1})$ for all $a\in\Ao$.
Let $\delo\colon\Ao\to\Ao$ be the mapping, defined by
$\delo(a)=\der(a)e+\con(a)\INT({1})$ for all $a\in\Ao$.
We see that $\con(\delo(a))=\der(a)$ and $\der(\delo(a))=\con(a)$
for each $a\in\Ao$ by hypotheses $(\ref{(1)})$, $(\ref{(4)})$ and $(\ref{(6)})$.
These two identities combined with $(\ref{(5)})$ show that $\delo$ is an isometry
on $B$.
Take $b\in\Ao$ arbitrarily.
Then $b=\con(b)e+\der(b)\INT({1})$ as proved above.
Setting $a_1=\der(b)e+\con(b)\INT({1})$,
we infer from hypotheses $(\ref{(1)})$, $(\ref{(4)})$ and $(\ref{(6)})$ that
$\con(a_1)=\der(b)$ and $\der(a_1)=\con(b)$.
Therefore, $\delo(a_1)=\der(a_1)e+\con(a_1)\INT({1})
=\con(b)e+\der(b)\INT({1})=b,$ and consequently,
$\delo$ is a surjective isometry on $B$.
For each $\la\in\T$, we have
$\delo(\la e)=\der(\la e)e+\con(\la e)\INT({1})=\la\,\INT({1})$
by the linearity of $\con$ and $\der$ and hypotheses $(\ref{(1)})$, $(\ref{(4)})$
and $(\ref{(6)})$.
Thus, $\con(\delo(\la e))=0$ by $(\ref{(4)})$, and therefore,
Lemma~\ref{lem4.2} does not hold for the surjective isometry
$\delo|_{\SAo}$ on $\SAo$ with $\cha=\set{x_0}$.

Although Lemma~\ref{lem4.2} does not hold if $\cha$ reduces to one point, it is enough for us to consider the case in which $\cha$ contains at least two points.
In fact, when $\cha$ contains just one point, we identify $A$ with $\C$ and define $\mathcal{U}\colon\Ao\to\ell_1^{(2)}$by $\mathcal{U}(a)=(\con(a),\der(a))$ for each $a\in\Ao$, where $\ell_1^{(2)}$ denotes the 2-dimensional $\ell_1$ space. Then $\mathcal{U}$ is a complex linear map from $B$ to $\ell_1^{(2)}$.
For each $(\la,\mu)\in\C^2$, set $b=\la e+\mu\,\INT(1)\in\Ao$.
We derive from hypotheses $(\ref{(1)})$, $(\ref{(4)})$ and $(\ref{(6)})$ that
$\con(b)=\la$ and $\der(b)=\mu$.
Hence, $\mathcal{U}(b)=(\la,\mu)$, which shows that
$\mathcal{U}$ is surjective.
We see from hypothesis $(\ref{(5)})$ that
$\mathcal{U}$ is an isometric isomorphism
form $(B,\V{\cdot}_B)$ onto $\ell_1^{(2)}.$ 
Tingley's problem for $\ell_1^{(2)}$, and hence
for $(B,\V{\cdot}_B)$, was solved as a special case of some results (see, for example, \cite{wan2} or \cite{Ding2004ell11}). Consequently, we can exclude the case in which $\cha$ reduces to a single point.
\end{rem}

In the rest of this paper, we assume that $\cha$  contains at least two
points, and thus Lemma~\ref{lem4.2} is valid. We can now complete the information
concerning the maps $\al$ and $\e$.

\begin{lem}\label{lem4.3}
The mappings $\al$ and $\e$ are constant on $\MT$.
\end{lem}

\begin{proof}
Fix arbitrary $\la\in\T$ and $\eta=(x,z)\in\M\times\T$.
Since $\la\t{e}\in\la V_\eta$, we derive from \eqref{lem4.1.1} that
\begin{align*}
1&=|\al(\la,\eta)|
=|\con(\del(\la\unitI))+\der(\del(\la\unitI))(\phio(x))\phit(\la,\eta)|\\
&\leq|
\con(\del(\la\unitI))|+|\der(\del(\la\unitI))(\phio(x))|\leq\Vs{\del(\la\unitI)}=1,
\end{align*}
and thus we get
\[
|\con(\del(\la\unitI))+\der(\del(\la\unitI))(\phio(x))\phit(\la,\eta)|=1
=|\con(\del(\la\unitI))|+|\der(\del(\la\unitI))(\phio(x))|.
\]
Since, by Lemma~\ref{lem4.2}, $\con(\del(\la\unitI))\neq0,$ we deduce that
\begin{equation}\label{lem4.3.1}
\der(\del(\la\unitI))(\phio(x))\phit(\la,\eta)=t\con(\del(\la\unitI))
\end{equation}
for some $t\geq0$.
The above equality shows that
\[
|t\con(\del(\la\unitI))|=|\der(\del(\la\unitI))(\phio(x))|=1-|\con(\del(\la\unitI))|,
\]
which yields $(1+t)|\con(\del(\la\unitI))|=1$.
It follows from \eqref{(4.1)}, \eqref{lem4.1.1} and \eqref{lem4.3.1} that
\begin{align*}
\la^{\e(\eta)}\al(\eta)
&=
\con(\del(\la\unitI))+\der(\del(\la\unitI))(\phio(x))\phit(\la,\eta)\\
&=
(1+t)\con(\del(\la\unitI))
=\fr{\con(\del(\la\unitI))}{|\con(\del(\la\unitI))|}.
\end{align*}
Letting $\la=1,i$ in the above equalities, we get
$\al(\eta)=\con(\del(\unitI))/|\con(\del(\unitI))|$ and
$i^{\e(\eta)}\al(\eta)=\con(\del(i\unitI))/|\con(\del(i\unitI))|$.
Since $\eta\in\MT$ is arbitrary, we have proved that
$\al$ and $\e$ are both constant maps on $\MT$.
\end{proof}

By Lemma~\ref{lem4.3}, we may, and do, write
$\al(\eta)=\al$ and $\e(\eta)=\e$ for all $\eta\in\MT$.
Under the light of this notation and \eqref{(4.3)}  we can rewrite \eqref{lem4.1.2} as
\begin{equation}\label{lem4.3.2}
\con(\del(a))+\der(\del(a))(\phio(x))\phit(\la,\eta)
=\la^{\e}\al
\end{equation}
for all $\la\in\T$, $\eta=(x,z)\in\M\times\T$ and $a\in\SAo$
with $\ta\in\la V_\eta$.

\begin{lem}\label{lem4.4}
If $x\in\M$, $\la\in\T$ and $a\in\SAo$ satisfy $\der(a)(x)=\la$,
then $\con(\del(a))=0,$ $\Vinf{\der(\del(a))}=1$, and
\begin{equation}\label{lem4.4.1}
\der(\del(a))(\phio(x))\phit(\la z,(x,z))=(\la z)^{\e}\al
\end{equation}
for all $z\in\T$.
\end{lem}

\begin{proof}
Let $z\in\T$.
Since $1=|\la|\leq\|\der(a)\|_\infty \le \|a\|_\Ao = 1$, it follows that $\con(a) = 0$. 
Thus we have $\ta\in\la zV_{(x,z)}.$ 
We derive from \eqref{lem4.3.2} that
\begin{equation}\label{lem4.4.2}
\con(\del(a))+\der(\del(a))(\phio(x))\phit(\la z,(x,z))=(\la z)^{\e}\al.
\end{equation}
Since $z\in\T$ is arbitrary, we must have $\Vinf{\der(\del(a))}\neq0$
by the above equality.
The triangle inequality 
shows that
\begin{align*}
1
&=
|\con(\del(a))+\der(\del(a))(\phio(x))\phit(\la z,(x,z))|\\
&\leq
|\con(\del(a))|+|\der(\del(a))(\phio(x))| \\
&\leq
|\con(\del(a))|+ \Vinf{\der(\del(a))} =\Vs{\del(a)}=1.
\end{align*}
These inequalities show that
$|\der(\del(a))(\phio(x))|=\Vinf{\der(\del(a))}\neq0$, and 
\begin{equation}\label{lem4.4.3}
\con(\del(a))=s\der(\del(a))(\phio(x))\phit(\la z,(x,z))
\end{equation}
for some $s\geq0$.
Then we have
\[
(1+s)\der(\del(a))(\phio(x))\phit(\la z,(x,z))=(\la z)^{\e}\al
\]
by \eqref{lem4.4.2}.
Taking the modulus on both sides of the above equality,
we obtain $(1+s)\Vinf{\der(\del(a))}=1$,
that is, $s\Vinf{\der(\del(a))}=1-\Vinf{\der(\del(a))}$.
We deduce from these equalities that
\[
\der(\del(a))(\phio(x))\phit(\la z,(x,z))
=\Vinf{\der(\del(a))}(\la z)^{\e}\al.
\]
Plugging the last identity into \eqref{lem4.4.3} we deduce that
\[
\con(\del(a))=s\Vinf{\der(\del(a))}(\la z)^{\e}\al=(1-\Vinf{\der(\del(a))})(\la z)^{\e}\al.
\]
Having in mind that $z\in\T$ is arbitrary, we have
$1-\Vinf{\der(\del(a))}=0=\con(\del(a))$.
Equality \eqref{lem4.4.2} shows that
$\der(\del(a))(\phio(x))\phit(\la z,(x,z))=(\la z)^{\e}\al$
for all $z\in\T$.
\end{proof}

Our next goal will consist in proving that the map $\phit(\la,(x,z))$ can be expressed as a product of single variable maps with $\la,z\in\T$ and $x\in\cha$.

\begin{lem}\label{lem4.5}
We set $\et(x)=\et(x,1)$ for each $x\in\cha$.
Then
\[
\phit(\la,(x,z))=\la^{\e-\et(x)}\phit(x,1)z^{\et(x)}
\]
for all $\la\in\T$ and $(x,z)\in\MT$.
\end{lem}

\begin{proof} Take any $\la,z\in\T$ and $x\in\M$. Pick a function $f_0\in \SA$ with $f_0(x) =1$ (we apply here that $A$ is extremely C-regular). Set $\mu=\la\ov{z}$ and $v=\mu f_0\in\SA$. Then we have $\mathcal{I}(v)\in\SAo$ and $(\der\circ\mathcal{I})(v)(x)=\mu$
by hypotheses $(\ref{(1)})$ and $(\ref{(5)})$ in Theorem~\ref{thm2}.
By applying \eqref{lem4.4.1} twice to $a=\mathcal{I}(v)$, we obtain
\begin{align*}
\der(\del(\mathcal{I}(v)))(\phio(x))\phit(\mu z,(x,z))
&=
(\mu z)^{\e}\al
=\mu^{\e}\al\cdot z^{\e}\\
&=
\der(\del(\mathcal{I}(v)))(\phio(x))\phit(\mu,(x,1))z^{\e}.
\end{align*}
The above equalities show that
$|\der(\del(\mathcal{I}(v)))(\phio(x))|=1$,
and hence we have $\phit(\mu z,(x,z))=\phit(\mu,(x,1))z^{\e}$.
By the choice of $\mu$, we get
\[
\phit(\la,(x,z))
=\phit(\la\ov{z},(x,1))z^{\e}.
\]
We derive from Lemma~\ref{lem4.3} combined with \eqref{(4.3)} that
\begin{align*}
\phit(\la,(x,z))
&=
\phit(\la\ov{z},(x,1))z^{\e}
=(\la\ov{z})^{\e-\et(x,1)}\phit(x,1)z^{\e}\\
&=
\la^{\e-\et(x,1)}\phit(x,1)z^{\et(x,1)}.
\end{align*}
We obtain
$\phit(\la,(x,z))=\la^{\e-\et(x)}\phit(x,1)z^{\et(x)}$
with $\et(x,1)=\et(x)$.
\end{proof}

For simplicity of notation, we shall write $\phit(x,1)=\phit(x)$
for each $x\in\cha$. It follows from Lemma~\ref{lem4.5} that \begin{equation}\label{lem4.5.1}
\phit(\la,(x,z))=\la^{\e-\et(x)}\phit(x)z^{\et(x)}
\end{equation}
for all $\la\in\T$ and $(x,z)\in\M\times\T$.
In particular, $\phit(x,z)=\phit(1,(x,z))=\phit(x)z^{\et(x)}$ for each $(x,z)\in\MT$
(see \eqref{(4.3)}).
Equality \eqref{lem4.3.2} is rewritten as
\begin{equation}\label{lem4.5.2}
\con(\del(a))
+\der(\del(a))(\phio(x))\la^{\e-\et(x)}\phit(x)z^{\et(x)}
=\la^{\e}\al
\end{equation}
for all $\la\in\T$, $(x,z)\in\M\times\T$ and $a\in\SAo$
with $\tf\in\la V_{(x,z)}$.

We next complete the concrete description of $\con\circ\del$ and $\der\circ\del$.

\begin{lem}\label{lem4.6}
Let  $\la\in\T$, $(x,z)\in\M\times\T$ and $a\in\SAo$ satisfy $\tf\in\la V_{(x,z)}$. Then
$$\begin{aligned}
\con(\del(a))& =|\con(\del(a))|\la^{\e}\al \q\mbox{ and }\\
\der(\del(a))(\phio(x))
& =\Vinf{\der(\del(a))}\la^{\et(x)}\al\ov{\phit(x)}z^{-\et(x)}.
\end{aligned}$$
In particular,
\begin{equation}\label{lem4.6.1}
|\con(\del(a))|+|\der(\del(a))(\phio(x))|
=|\con(a)|+|\der(a)(x)|
\end{equation}
for all $a\in\SAo$ with $\tf\in\la V_{(x,z)}$.
\end{lem}

\begin{proof}
We notice that \eqref{lem4.5.2} is valid by assumption.
Then we obtain
\begin{align}\label{lem4.6.2}
1
&\leq
|\con(\del(a))|+|\der(\del(a))(\phio(x))\la^{\e-\et(x)}\phit(x)z^{\et(x)}|\\
&\leq
|\con(\del(a))|+\Vinf{\der(\del(a))}=\Vs{\del(a)}=1.
\notag
\end{align}
We have $|\der(\del(a))(\phio(x))|=\Vinf{\der(\del(a))}$
by the above inequalities.

We first consider the case in which $\con(\del(a))=0$.
Then the identity $\con(\del(a))=|\con(\del(a))|\la^{\e}\al$ is obvious,
and thus $\Vinf{\der(\del(a))}=\Vs{\del(a)}=1$. Then we have
$\Vinf{\der(\del(a))}\la^{\et(x)}\al\ov{\phit(x)}z^{-\et(x)}
=\der(\del(a))(\phio(x))$
by \eqref{lem4.5.2}.

Now we consider the case in which $\con(\del(a))\neq0$.
Since equalities hold in \eqref{lem4.6.2}, there exists $s\geq0$ such that
$\der(\del(a))(\phio(x))\la^{\e-\et(x)}\phit(x)z^{\et(x)}=s\con(\del(a))$.
By applying the last equality to \eqref{lem4.5.2}, we have
$(1+s)\con(\del(a))=\la^{\e}\al$, and thus $(1+s)|\con(\del(a))|=1$.
Then $\con(\del(a))=|\con(\del(a))|\la^{\e}\al$ is valid even if $\con(\del(a))\neq0$.
We deduce from \eqref{lem4.5.2} combined with
$|\con(\del(a))|+\Vinf{\der(\del(a))}=1$
that
$$\begin{aligned}
\der(\del(a))(\phio(x))\la^{\e-\et(x)}\phit(x)z^{\et(x)}
=\la^{\e}\al-\con(\del(a))\\
=(1-|\con(\del(a))|)\la^{\e}\al
=\Vinf{\der(\del(a))}\la^{\e}\al.
\end{aligned}$$
Consequently,
$\der(\del(a))(\phio(x))
=\Vinf{\der(\del(a))}\la^{\et(x)}\al\ov{\phit(x)}z^{-\et(x)}$.

We shall finally  prove \eqref{lem4.6.1}.
Having in mind that $\tf\in\la V_{(x,z)}$, we have
\[
1=|\la|=|\con(a)+\der(a)(x)z|\leq|\con(a)|+|\der(a)(x)|
\leq\VB{a}=1.
\]
It follows from \eqref{lem4.6.2} that
$|\con(\del(a))|+|\der(\del(a))(\phio(x))|=1=|\con(a)|+|\der(a)(x)|$.
\end{proof}

Let $\la\in\T$ and $x\in\M$.
We define a subset $\la P_x$ of $\SA$ by
\[
\la P_x=\set{f\in\SA:f(x)=\la}.
\] Clearly $P_x$ is non-empty by the assumptions on $A$.

Our next result is a generalization of \cite[Lemma~2.17]{cue}
for extremely C-regular spaces.

\begin{lem}\label{lem4.7}
For each $x_0\in\M$ and $a\in\SAo$,
let $\la\in\T$ be such that $\der(a)(x_0)=|\der(a)(x_0)|\la$.
Then for each $t\in\R$ with $0<t<1$, and $f_0\in  P_{x_0}$ there exists 
$g_t\in P_{x_0}$ such that
$$ |t\con(a)|\la f_0 +t\der(a)+\left(1-|t\con(a)|-|t\der(a)(x_0)|\right) \la g_t \in\la P_{x_0}.$$
\end{lem}

\begin{proof}
We notice that $1-|t\con(a)|-|t\der(a)(x_0)|\geq1-t>0$,
since $|\con(a)|+|\der(a)(x_0)|\leq\VB{a}=1$.
Setting $r=1-|t\con(a)|-|t\der(a)(x_0)|$, we have $r>0$.
We define
\begin{align*}
G_0
&=
\left\{x\in X:|t\der(a)(x)-t\der(a)(x_0)|\geq\fr{r}{4}\right\},\\
G_m
&=
\left\{x\in X:\fr{r}{2^{m+2}}
\leq|t\der(a)(x)-t\der(a)(x_0)|\leq\fr{r}{2^{m+1}}\right\}
&(m\in\N).
\end{align*}
By the continuity of $\der(a)$, we observe that
$G_n$ is a closed subset of $X$ with $x_0\not\in G_n$ for all
$n\in\N\cup\set{0}$.
Since $A$ is extremely C-regular, for each $n\in\N$ there exists
$f_n\in P_{x_0}$ such that
\begin{equation}\label{lem4.7.1}
|f_n|<\fr{1-t}{2r}
\q\mbox{on $G_0\cup G_n$}.
\end{equation}
We denote by $g_t$ the limit of the convergent series
$\sum_{n=1}^\infty f_n/2^n$.
Having in mind that $f_n\in P_{x_0}$ for each $n\in\N$, we obtain
$1=g_t(x_0)\leq\Vinf{g_t}\leq\sum_{n=1}^\infty\Vinf{f_n}/2^n=1$.
Hence $g_t(x_0)=\Vinf{g_t}=1$, which yields $g_t\in P_{x_0}$.
We set
\[
h_t=|t\con(a)|\la  f_0+t\der(a)+r\la g_t\in A.
\]
We shall prove $h_t\in\la P_{x_0}$.
Since $t\der(a)(x_0)=|t\der(a)(x_0)|\la$, we get
$|t\con(a)|\la+t\der(a)(x_0)=(1-r)\la$.
Now we obtain
\[
h_t(x_0)=|t\con(a)|\la+t\der(a)(x_0)+r\la g_t(x_0)
=(1-r)\la+r\la=\la,
\]
which proves $h_t(x_0)=\la$.
We need to prove that $\Vinf{h_t}=1$.
Fix an arbitrary $x\in X$ to show that $|h_t(x)|\leq1$.
We shall consider three different cases.
First, if $x\in G_0$, then $|f_n(x)|<(1-t)/(2r)$
for each $n\in\N$ by \eqref{lem4.7.1},
which implies $|g_t(x)|<(1-t)/r$ by definition.
Consequently, we have
\[
|h_t(x)|\leq|t\con(a)|+|t\der(a)(x)|+|rg_t(x)|
<\VB{ta}+1-t=1,
\]
and hence $|h_t(x)|<1$ if $x\in G_0$.

Second, we consider the case in which $x\in\cup_{n=1}^\infty G_n$,
and then $x\in G_m$ for some $m\in\N$.
By the choice of $G_m$, we have $|t\der(a)(x)-t\der(a)(x_0)|\leq r/2^{m+1}$,
and thus $|t\der(a)(x)|\leq|t\der(a)(x_0)|+r/2^{m+1}$.
Applying \eqref{lem4.7.1}, we obtain
\[
|r\la g_t(x)|\leq
r\left(\fr{|f_m(x)|}{2^m}
+\sum_{n\neq m}\fr{|f_n(x)|}{2^n}\right)
\leq\fr{1-t}{2^{m+1}}+r\left(1-\fr{1}{2^m}\right).
\]
By the definition of $r$, we can write
$1-r=|t\con(a)|+|t\der(a)(x_0)|\leq t$, which leads to
\begin{align*}
|h_t(x)|
&\leq
|t\con(a)|+|t\der(a)(x)|+|r\la g_t(x)|\\
&\leq
|t\con(a)|+|t\der(a)(x_0)|+\fr{r}{2^{m+1}}
+\fr{1-t}{2^{m+1}}+r\left(1-\fr{1}{2^m}\right)\\
&=
1-r+\fr{1-t}{2^{m+1}}+r\left(1-\fr{1}{2^{m+1}}\right)\\
&=
1+\fr{1-r-t}{2^{m+1}}\leq1.
\end{align*}
We thus obtain $|h_t(x)|\leq1$ for $x\in\cup_{n=1}^\infty G_n$.

Finally, if $x\not\in\cup_{n=0}^\infty G_n$,
then $\der(a)(x)=\der(a)(x_0)$.
Thus, $|h_t(x)|\leq|t\con(a)|+|t\der(a)(x_0)|+r=1$.
We have therefore proved that $|h_t(x)|\leq1$ for all $x\in X$.
\end{proof}

\section{Proof of main results}\label{sect5}

We have already gathered all the tools to prove Theorem~\ref{thm2}. The main idea of the following proof is based on the proof of \cite[Theorem~2.1]{cue}.

\begin{proof}[\bf Proof of Theorem~\ref{thm2}]
Take any $a\in\SAo,$ $x\in\M$ and $f_0\in P_x$.
We can choose $\la\in\T$ so that $\der(a)(x)=|\der(a)(x)|\la$.
Fixing arbitrary $t\in\R$ with $0<t<1$,
we set $r=1-|t\con(a)|-|t\der(a)(x)|$.
Lemma~\ref{lem4.7} assures the existence of an element $g_t\in P_x$
such that $h_t=|t\con(a)|\la  f _0 +t\der(a)+r\la g_t\in\la P_x$.
Then we have
\begin{align*}
\Vinf{h_t-\der(a)}
&=
\Vinf{|t\con(a)|\la {f _0}+(t-1)\der(a)+r\la g_t}\\
&\leq
|t\con(a)|+(1-t)\Vinf{\der(a)}+1-|t\con(a)|-|t\der(a)(x)|\\
&=
(1-t)\Vinf{\der(a)}+1-|t\der(a)(x)|.
\end{align*}
Having in mind that $h_t\in\la P_x$, we derive from hypotheses $(\ref{(1)})$ and $(\ref{(4)})$ that $\der(\INT(h_t))(x)=h_t(x)=\la$
and $\con(\INT(h_t))=0$.
This implies that $\t{\INT(h_t)}\in\la V_{(x,1)}$ by \eqref{tilde}.
We apply Lemma~\ref{lem4.4} 
to the element $\INT(h_t)$,
then we get $\con(\del(\INT(h_t)))=0$.
Setting $\zeta=\phio(x)$, we infer from \eqref{lem4.5.2} that
\[
\der(\del(\mathcal{I}(h_t)))(\zeta)=\der(\del(\mathcal{I}(h_t)))(\phio(x))
=\la^{\et(x)}\al\ov{\phit(x)}.
\]
By combining these equalities we get
\begin{align*}
1-|\der(\del(a))(\zeta)|
&=
|\la^{\et(x)}\al\ov{\phit(x)}|-|\der(\del(a))(\zeta)|\\
&\leq
|\la^{\et(x)}\al\ov{\phit(x)}-\der(\del(a))(\zeta)|\\
&=
|\der(\del(\INT(h_t)))(\zeta)-\der(\del(a))(\zeta)|\\
&\leq
\Vinf{\der(\del(\INT(h_t)))-\der(\del(a))}\\
&=
\Vs{\del(\INT(h_t))-\del(a)}-|\con(\del(a))|\\
&=
\Vs{\INT(h_t)-a}-|\con(\del(a))|\\
&=
|\con(a)|+\Vinf{h_t-\der(a)}-|\con(\del(a))|\\
&\leq
|\con(a)|+(1-t)\Vinf{\der(a)}+1-|t\der(a)(x)|-|\con(\del(a))|,
\end{align*}
where we have used that $\con(\del(\INT(h_t)))=0=\con(\INT(h_t))$
and $\del$ is an isometry with respect to $\VB{\cdot}$.
Taking $t\to1$ in the above inequalities, we get
\begin{equation}\label{proof1}
\begin{aligned}
1-|\der(\del(a))(\zeta)|
&\leq|\la^{\et(x)}\al\ov{\phit(x)}-\der(\del(a))(\zeta)|\\
&\leq|\con(a)|+1-|\der(a)(x)|-|\con(\del(a))|.
\end{aligned}
\end{equation}
Having in mind that $\zeta=\phio(x)$, we especially have
\begin{equation}\label{proof2}
|\con(\del(a))|-|\der(\del(a))(\phio(x))|\leq|\con(a)|-|\der(a)(x)|.
\end{equation}
Now we take $x_0\in\cha$ so that $|\der(a)(x_0)|=\Vinf{\der(a)}$.
Since $x\in\cha$ was arbitrary, we obtain
\begin{equation}\label{proof3}
|\con(\del(a))|-|\der(\del(a))(\phio(x_0))|\leq|\con(a)|-\Vinf{\der(a)}.
\end{equation}
Let $\mu,z\in\T$ be such that $\con(a)=|\con(a)|\mu$ and
$\der(a)(x_0)=|\der(a)(x_0)|z=\Vinf{\der(a)}z$.
We have $\ta\in\mu V_{(x_0,\ov{z}\mu)}$, since
\[
\con(a)+\der(a)(x_0)\ov{z}\mu = \ta (x_0,\ov{z}\mu)
=(|\con(a)|+\Vinf{\der(a)})\mu=\Vs{a}\cdot\mu=\mu.
\]
We can apply \eqref{lem4.6.1} to $a\in\SAo$, and then we have
\[
|\con(\del(a))|+|\der(\del(a))(\phio(x_0))|=|\con(a)|+\Vinf{\der(a)}.
\]
Adding the last equality to \eqref{proof3}, we get
$|\con(\del(a))|\leq|\con(a)|$.
Since $\del$ is a surjective isometry, we see that $\deli$ enjoys the same property we proved for  $\del$, and thus, $|\con(\deli(b))|\leq|\con(b)|$ for all $b\in\SAo$, which, particularly, yields $|\con(a)|\leq|\con(\del(a))|$.
Consequently, we obtain $|\con(\del(a))|=|\con(a)|$, and then
we derive from \eqref{proof2} that
\begin{equation}\label{proof4}
|\der(a)(x)|\leq|\der(\del(a))(\phio(x))|.
\end{equation}
For each $\nu\in\C$, we define $[\nu]^\varepsilon=\nu$
if $\varepsilon=1$ and $[\nu]^\varepsilon=\ov{\nu}$ if $\varepsilon=-1$.
Noting that $\ta\in\mu V_{(x_0,\ov{z}\mu)}$,
we deduce from Lemma~\ref{lem4.6} that
$\con(\del(a))=|\con(\del(a))|\mu^{\e}\al=|\con(a)|\mu^{\e}\al
=[\con(a)]^{\e}\al$, that is,
\begin{equation}\label{proof5}
\con(\del(a))=[\con(a)]^{\e}\al.
\end{equation}

We shall next show that $\phio$ is injective.
Let $x_1,x_2\in\cha$, and assume that $\phio(x_1)=\phio(x_2)$. Having in mind \cite[Proposition~5.4]{hat3}, we can find $f_1\in \SA$ such that $f_1(x_j)=1$ for all $j=1,2$.
Setting $a_1=\mathcal{I}(f_1)$, we see that 
$\der(a_1)(x_j)=1$ for $j=1,2$ by $(\ref{(1)})$.
Lemma~\ref{lem4.4} shows that
$\con(\del(a_1))=0$.
Since $\con(a_1)=0$ by $(\ref{(4)})$,
we obtain $\t{a_1}\in V_{(x_j,1)}$
for $j=1,2$.
By applying \eqref{lem4.5.2} with $\la=1=z$,
we have $$\der(\del(a_1))(\phio(x_j)) \phit(x_j) =\al \hbox{ for } j=1,2.$$
The last equality shows that
$\der(\del(a_1))(\phio(x_j))\neq0$, which yields $\phit(x_1)  = \phit(x_2),$ since $\phio (x_1) = \phio (x_2)$.  Note that $T(V_{(x_j,1)})
=\al V_{(\phio(x_j),\phit(x_j))}$
 (cf. \eqref{lem4.1.3} and Lemma~\ref{lem4.3}), which consequently implies $$T(V_{(x_1,1)})
 =\al V_{(\phio(x_1),\phit(x_1))} 
 =\al V_{(\phio(x_2),\phit(x_2))}= T(V_{(x_2,1)}),$$ equivalently, $V_{(x_1,1)}=V_{(x_2,1)}$. It follows that  $x_1=x_2$ by Lemma~\ref{lem2.1},
and consequently $\phio$ is injective.

Since $\phio$ is bijective, we can apply the arguments
leading to \eqref{proof4} to $\deli$ and $\phio^{-1}$ instead of
$\del$ and $\phio$.
Then we obtain
$|\der(b)(y)|\leq|\der(\deli(b))(\phio^{-1}(y))|$
for $b\in\SAo$ and $y\in\cha$.
Taking $b=\del(a)$ and $y=\phio(x)$, we conclude
$|\der(\del(a))(\phio(x))|\leq|\der(a)(x)|$, which yields
$$|\der(\del(a))(\zeta)|=|\der(\del(a))(\phio(x))|=|\der(a)(x)|.$$

Since $|\con(\del(a))|=|\con(a)|$, we derive from
\eqref{proof1} that
\[
|\la^{\et(x)}\al\ov{\phit(x)}-\der(\del(a))(\zeta)|+|\der(\del(a))(\zeta)|=1.
\]
This implies that equality holds in a triangle inequality.
Hence, there exists $s\geq0$ such that
$\der(\del(a))(\zeta)=s\la^{\et(x)}\al\ov{\phit(x)}$.
Then
$s=|s\la^{\et(x)}\al\ov{\phit(x)}|=|\der(\del(a))(\zeta)|=|\der(a)(x)|$.
Having in mind that $\der(a)(x)=|\der(a)(x)|\la$, we obtain
$\der(\del(a))(\zeta)
=|\der(a)(x)|\la^{\et(x)}\al\ov{\phit(x)}
=[\der(a)(x)]^{\et(x)}\al\ov{\phit(x)}$,
and hence, by the arbitrariness of $a \in \SAo$ and $x\in\cha,$ we conclude that 
\begin{equation}\label{proof6}
\der(\del(a))(\phio(x))
=\al\ov{\phit(x)}\,[\der(a)(x)]^{\et(x)}
\end{equation} for all $a \in \SAo$ and $x\in\cha$.

Finally, we shall prove that $\del\colon\SAo\to\SAo$ extends to a surjective real linear isometry $\del_0\colon\Ao\to\Ao$ via the standard positive homogeneous extension.
Let us consider this positive homogeneous extension
$\del_0\colon\Ao\to\Ao$ defined by
\[
\del_0(b)=
\begin{cases}
\D\Vs{b}\,\del\left(\fr{b}{\Vs{b}}\right), &b\in\Ao\setminus\set{0},\\[2mm]
0, &b=0.
\end{cases}
\]
Since $\con$ is a linear map, we derive from \eqref{proof5} that
\[
\con(\del_0(b))
=\Vs{b}\,\con\left(\del\left(\fr{b}{\Vs{b}}\right)\right)
=\Vs{b}\left[\con\left(\fr{b}{\Vs{b}}\right)\right]^{\e}\al
=[\con(b)]^{\e}\al
\]
for all $b\in B\sm\set{0}$.
Having in mind that $\con(\del_0(0))=\con(0)=0$, we have
$\con(\del_0(b))=[\con(b)]^{\e}\al$ for all $b\in B$. By a quite similar argument, we obtain from  \eqref{proof6} that 
$\der(\del_0(b))(\phio(x))=\al\ov{\phit(x)}[\der(b)(x)]^{\et(x)}$
for all $b\in B$ and $x\in\cha$. By combining all the previous conclusions we arrive at 
$$\begin{aligned}
\Vs{\del_0(b_1)-\del_0(b_2)} &=  |\con(\del_0(b_1))-\con(\del_0(b_2))| \\
&+ \| \der(\del_0(b_1))-\der(\del_0(b_2)) \|_{\infty} \\
&= |\con(\del_0(b_1))-\con(\del_0(b_2))| \\
&+\sup_{x\in\M}|\der(\del_0(b_1))(\phio(x))-\der(\del_0(b_2))(\phio(x))|\\
&=
|\con(b_1)-\con(b_2)|\\
&+\sup_{x\in\M}|\der(b_1)(x)-\der(b_2)(x)|
=\Vs{b_1-b_2}
\end{aligned}$$ for all $b_1,b_2\in\Ao$, where we have used that
$\cha$ is a boundary for $A$ and $\phio$ satisfies $\phio(\cha)=\cha$. This shows that $\del_0$ is an isometry on $\Ao$ with respect to the norm $\VB{\cdot}$ (and obviously surjective). The Mazur--Ulam theorem \cite{maz} implies that $\del_0$ is real linear.
Therefore, $\del_0$ is a surjective real linear isometry extending $\del$ to $\Ao$.
\end{proof}

\begin{proof}[\bf Proof of Theorem~\ref{thm1}]  We assume that we are now under the hypotheses of Theorem~\ref{thm1}.

Let $\widehat{B}$ be the completion of $B$ with norm $\V{\cdot}_{\widehat{B}}$.
It is part of the folklore of the theory, and routine to check that there exists a unique surjective isometry
$\widehat{\del}\colon S_{\widehat{B}}\to S_{\widehat{B}}$
whose restriction to $\SAo$ is $\del\colon\SAo\to\SAo$.
In addition, we see that $\del$ is extended to a surjective real
linear isometry on $B$ if and only if $\widehat{\del}$ is extended
to a surjective real linear isometry on $\widehat{B}$.
We shall show that the Banach space $\widehat{B}$ satisfies the conditions
in Theorem~\ref{thm2}.

Let $X_\infty$ be the one point compactification of $X$, that is,
$X_\infty=X\cup\set{\infty}$.
We define $f(\infty)=0$ for each $f\in A$.
Then we may, and do, regard $A$ as a subspace of $C(X_\infty)$.
Let $\ov{A}$ be the uniform closure of $A$ in $C(X_\infty)$.
We show that $\ov{A}$ is extremely C-regular.
Let $A^*$ and $\ov{A}^*$ be the dual spaces of $A$ and $\ov{A}$, respectively.
Then the correspondence from $\ov{A}^*$ to $A^*$, defined by $\chi\mapsto\chi|_A$
for each $\chi\in\ov{A}^*$, is an isometric isomorphism.
We observe that the set of all extreme points of the closed unit ball
$\ov{A}^*_1$ coincides with that of $A^*_1$.
Thus, $\d_x$ is an extreme point of $\ov{A}^*_1$ if and only if $\d_x|_A$ is
an extreme point of $A^*_1$ for $x\in X_\infty$.
This implies that $\operatorname{Ch}(\ov{A})=\cha$. We see that $\ov{A}$ is an extremely C-regular subspace of $C(X_\infty)$ since so is $A$.

Note that $A$ and $\Ao$ are dense in $\ov{A}$ and $\widehat{\Ao}$,
respectively.
It is routine to check that there exist unique continuous extensions
$\widehat{\INT}\colon\ov{A}\to\widehat{\Ao}$,
$\widehat{\der}\colon\widehat{\Ao}\to\ov{A}$ and
$\widehat{\con}\colon\widehat{\Ao}\to\C$
of the bounded linear operators
$\INT$, $\der$ and $\con$, respectively.

We shall show that the quintuple $(\ov{A},\widehat{\Ao},\widehat{\INT},
\widehat{\der},\widehat{\con})$ satisfies the hypotheses in Theorem~\ref{thm2}.
Since $(\widehat{\der} \circ \widehat{\INT})(f)=(\der\circ\INT)(f)=f$
for each $f\in A$, we observe that $(\widehat{\der} \circ \widehat{\INT})(h)=h$
for each $h\in\ov{A}$.
Hence, hypothesis $(\ref{(1)})$ in Theorem~\ref{thm2} holds for $\widehat{\der}\circ\widehat{\INT}$.
By a quite similar argument, we see that
$\widehat{\con}\circ\widehat{\INT}$ satisfies hypothesis $(\ref{(4)})$  in Theorem~\ref{thm2}.
By $(\ref{(5)})$, we have $\V{a}_{\widehat{\Ao}}=\V{a}_B=|\con(a)|+\Vinf{\der(a)}
=|\widehat{\con}(a)|+\Vinf{\widehat{\der}(a)}$ for $a\in B$.
We thus obtain
$\V{b}_{\widehat{\Ao}}=|\widehat{\con}(b)|+\Vinf{\widehat{\der}(b)}$ for
all $b\in\widehat{\Ao}$,
and thus $(\widehat{\con},\widehat{\der})$ satisfies hypothesis $(\ref{(5)})$  in Theorem~\ref{thm2}.
By assumption, there exists $e\in\ker\der$
such that $\con(e)=\V{e}_B=1$.
By the definition of $\widehat{\der}$ and $\widehat{\con}$, we see that
$e\in\ker\widehat{\der}$ and $\widehat{\con}(e)=\V{e}_{\widehat{\Ao}}=1$.
We conclude that the quintuple
$(\ov{A},\widehat{\Ao},\widehat{\INT},\widehat{\der},\widehat{\con})$ satisfies
the hypotheses $(\ref{(1)})$ through $(\ref{(6)})$
in Theorem~\ref{thm2}.

Theorem~\ref{thm2} now implies that
$\widehat{\del}\colon S_{\widehat{B}}\to S_{\widehat{B}}$ extends to a surjective real linear isometry on $\widehat{B}$. We conclude that each surjective isometry $\del$ on $\SAo$ admits a surjective real linear extension to $\Ao$.
\end{proof}

\section{Examples}\label{sect6}

In this section, we give some examples of complex normed spaces $\Ao$ that satisfy
the hypotheses $(\ref{(1)})$ through $(\ref{(6)})$ in Theorem~\ref{thm1}.
Thus, every surjective isometry $\del\colon\SAo\to\SAo$ admits a surjective real linear isometric extension $\del_0\colon\Ao\to\Ao$ for those $B$'s.

\begin{ex}
Let $B$ be the linear space $\Lipo$ of all Lipschitz continuous
complex-valued functions on $[0,1]$.
We denote by $L(a)$ the Lipschitz constant for $a\in B$, that is,
\[ L(a)=\sup_{\substack{x,y\in[0,1]\\x\neq y}} \left\{\frac{|a(x) - a(y)|}{|x - y|}\right\} < \infty.
\]
Fix an arbitrary $c \in [0,1]$, and define a norm on $B$ by
\[
\VB{a}=|a(c)| + L(a) \qquad (a \in B).
\]
Let $L^\infty([0,1])$ be the commutative $C^*$-algebra
of all bounded Lebesgue measurable functions on $[0,1]$.
Denote by $\MIS$ the maximal ideal space of $L^\infty([0,1])$.
We set $A=C(\MIS)$, and then $A$ is an extremely C-regular space.
The Gelfand transformation $\Gamma\colon L^\infty([0,1])\to A$ is
an isometric algebra isomorphism.
It is well-known, by the celebrated Rademacher's theorem, that every $a \in B=\Lipo$ is differentiable almost everywhere
with respect to the Lebesgue measure and that the derivative $a'$ of $a$
belongs to $L^\infty([0,1])$ and satisfies
\[
a(t) = a(c) + \int_c^t a'(x)\,dx
\]
for every $t \in [0,1]$, and $L(a)=\Vinf{a'}$.
Conversely, we see that the function $a_g$, defined by
$a_g(t)=\int_c^tg(x)\,dx$ for each $g\in L^\infty([0,1])$ and
$t\in[0,1]$, satisfies $a_g\in\Lipo$ and $L(a_g)=\Vinf{g}$
(see \cite[Example 1.6.5]{wea}).

We define bounded linear operators $\INT\colon A\to B$,
$\der\colon B\to A$ and $\con\colon B\to\C$ by
\[
\INT(f)(t)=\int_c^t\Gamma^{-1}(f)(x)\,dx
\qq(f\in A,\,t\in[0,1]),
\]
$\der(a)=\Gamma(a')$ and $\con(a)=a(c)$ for $a\in B$, respectively.
We observe that $(\der\circ\INT)(f)=\Gamma(\INT(f)')=f$
and $(\con\circ\INT)(f)=0$ for every $f\in A$.
Hence the conditions $(\ref{(1)})$ and $$(\ref{(4)})$$ in Theorem~\ref{thm1}
are satisfied.
Since $\Gamma$ is an isometry, we obtain
$L(a)=\Vinf{a'}=\Vinf{\Gamma(a')}=\Vinf{\der(a)}$ for each $a\in B$.
This shows that $\VB{a}=|\con(a)|+\Vinf{\der(a)}$ for all $a\in B$,
which proves $(\ref{(5)})$ in Theorem~\ref{thm1}.
Let $\unit$ be the constant function on $[0,1]$, which takes
only the value $1$. Clearly, $\unit\in\ker\der$ and $\con(\unit)=\VB{\unit}=1$.
Consequently, all the hypotheses in Theorem~\ref{thm1} are true for $B=\Lipo$.
\end{ex}

\begin{ex}
Fix any $c\in[0,1]$.
Let $B$ be the complex Banach space $C^1([0,1])$ of all continuously differentiable
complex-valued functions on $[0,1]$ with the norm
\[
\VB{a}=|a(c)|+\Vinf{a'}\qquad (a\in B).
\]
Set $A=C([0,1])$ and define bounded linear operators $\INT\colon A\to B$, $\der\colon B\to A$
and $\con\colon B\to\C$ by 
\[
\INT(f)(t) = \int_c^t f(x)\,dx, \quad 
\der(a) = a', \quad \text{and} \quad 
\con(a) = a(c),
\]
for every $f \in A$, $a\in B$ and $t \in [0,1]$. 
It is easy to check that
the quintuple $(A,B,\INT,\der,\con)$ fulfills
all the conditions in Theorem~\ref{thm1} with $e=\unit$ the unit element in $A$.
\end{ex}

\begin{ex}\label{ex3}
Let $\Di$ be the open unit disk $\{z\in\C:|z|<1\}$.
We denote by $H(\Di)$ the algebra of all analytic functions on $\Di$.
Let $H^\infty(\Di)$ be the commutative Banach algebra
of all bounded analytic functions on $\Di$ with the supremum norm. 
Let $c\in\Di$, and set $B=\set{a\in H(\Di):a'\in H^\infty(\Di)}$ with the norm
\[
\VB{a}=|a(c)|+\Vinf{a'}
\qq(a\in B).
\]
Let $\MIS$ be the maximal ideal space of $H^\infty(\Di)$.
Since the Gelfand transformation $\Gamma\colon H^\infty(\Di)\to C(\MIS)$
is an isometric homomorphism,
we see that $A=\Gamma(H^\infty(\Di))$ is a uniform algebra on $\MIS$.
Hence, $A$ is an extremely C-regular subspace of $C(\MIS)$. For uniform algebra $A$, it is known that $\cha$
is the set of all weak peak points for $A$
(see, for example \cite[Corollary~3.7 and Theorem~3.27]{hat3}).
Then one can check that $A$ is an extremely C-regular closed
subspace of $C(\MIS)$. Define bounded linear operators $\INT\colon A\to B$,
$\der\colon B\to A$ and $\con\colon B\to\C$ by
\[
\INT(f)(z)=\int_{[c,z]}\Gamma^{-1}(f)(\zeta)\,d\zeta, \quad 
\der(a)=\Gamma(a') \quad \text{and} \quad 
\con(a)=a(c)
\]
for every $f\in A$, $a\in B$ and $z\in\Di$.
We see that the  quintuple $(A,B,\INT,\der,\con)$ satisfies the conditions
$(\ref{(1)})$ through $(\ref{(6)})$ in Theorem~\ref{thm1} with $e=\unit$ the unit element in $A$.
\end{ex}

\begin{ex}
Let $A(\Db)$ be the disk algebra, that is, the commutative Banach algebra
of all analytic functions on $\Di$, which are extended to continuous functions on the closed unit disk $\Db$ with the supremum norm. Set
\[
B=\{a \in H(\Di):a'\in A(\Db)\ \mbox{and}\ a'(0)=0\}.
\]
For any $c\in\Di$, $B$ is a Banach space with the norm
$\VB{a}=|a(c)|+\Vinf{a'}$ for $a\in B$.
By the same reasoning as in Example~\ref{ex3},
we see that the Banach space $(B,\VB{\cdot})$
satisfies the conditions in Theorem~\ref{thm1}.
\end{ex}


\noindent\textbf{Acknowledgements:} We would like to thank the anonymous referee for pointing out a gap in the arguments in a former version of this paper.


\end{document}